\documentclass[twoside, 12pt, a4paper, reqno]{article}
\usepackage{amssymb, amsmath, amsthm, mathrsfs}
\usepackage{graphicx, subfig}
\usepackage{cite, hyperref}
\usepackage{type1cm}
\usepackage{xcolor} % \textcolor{blue}{text}

\hypersetup{colorlinks=true,linkcolor=red,citecolor=blue,filecolor=magenta,urlcolor=cyan}
\makeatletter
\renewcommand\@makefntext[1]{\noindent #1}
\makeatother
\usepackage[headsep=0.25truein, top=1.25truein, bottom=1truein, hmargin={1truein,1truein}]{geometry}
\allowdisplaybreaks[4]
\usepackage[rm,small,center]{titlesec}
\titlelabel{\thetitle.\enspace}
\usepackage{fancyhdr}
\numberwithin{equation}{section}
\numberwithin{figure}{section}

\theoremstyle{definition}
\newtheorem{defn}{Definition}[section]
\newtheorem{ex}{Example}
\newtheorem{rmk}{Remark}[section]

\theoremstyle{plain}
\newtheorem{prop}{Proposition}[section]
\newtheorem{lem}{Lemma}[section]
\newtheorem{thmm}{Theorem}

\newtheorem{thm}{Theorem}[section]
\newtheorem{cor}{Corollary}[section]
\usepackage{xpatch}
\xpatchcmd{\proof}{\itshape}{\normalfont\proofnamefont}{}{}
\newcommand{\proofnamefont}{\itshape}
\titleformat{\section}{\normalfont\large\bfseries}{\S\thesection}{1em}{}

\begin{document}
\fontsize{15pt}{16pt}\selectfont
\vspace*{-0.75truein}
\fancyhead[CE]{\footnotesize Wu-Hsiung Huang} % Authors' Name
\fancyhead[CO]{\footnotesize Singularities and Topological Change for Deforming Domains in Manifolds} % Paper Title
\setcounter{section}{0}
\setcounter{footnote}{0}
\setcounter{page}{1}
\thispagestyle{plain}

\vspace*{0.5truein}%
\footnotetext{%
Date: 2026-06-26 
}
%%===================================================================== Body ==
\begin{center} % Paper Title
\large \textbf{Singularities and Topological Changes\\of Deforming Domains in Manifolds}
\end{center}

\begin{center} % Authors' Name
Wu-Hsiung Huang
\end{center}

\vspace{0.1truein}%
\begin{quotation}
\noindent
Abstract. The main purpose of this paper is to investigate when topology jumps while analysis remains continuous. Given a $C^{0}$-deformation of domains $D(t)$ in a manifold $M^{n}$, allowing the topological type of the domains $D(t)$ to vary with $t$, in what situations are the analysis notions associated with $D(t)$  continuous in $t$, so that the machinery of analysis is still working along the deformation? This type of problem arose in our previous work \cite{H24} for domains on constant mean curvature (CMC) hypersurfaces in $\mathbb{R}^{n+1}$. In the present paper, we consider a general setting in which the deforming domains are situated in an arbitrary smooth manifold $M^{n}$ equipped with a self-adjoint strongly elliptic operator $L$, replacing the stability operator for CMC hypersurfaces in $\mathbb R^{n+1}$ considered in \cite{H24}.

We introduce the notion of quasi-Lipschitz domains by gluing certain boundary points of a Lipschitz domain in a specific manner, thereby allowing the topology of the deforming domain $D(t)$ to  change. The continuity theorems and the existence of the required deformations are proved.
We establish that any monotone $C^{0}$-deformation of quasi-Lipschitz domains in $M^{n}$ satisfies Sobolev continuity and eigenvalue continuity for the operator $L$ along the deformation parameter $t$. As a consequence, a \emph{global} Morse index theorem is obtained. Furthermore, given any quasi-Lipschitz domain $D$ in $M^{n}$, we construct a $C^0$-deformation from a small $n$-ball to the domain $D$, along which the topology of $D(t)$ may change, while the required continuity properties remain valid, and the  Morse index theorem holds for the deformation.
\end{quotation}

%%=================================================================== sec. 1 ==
\section{Introduction}
%%=============================================================================
Let $M^{n}$ be an oriented \emph{smooth} (i.e., $C^{\infty}$-) manifold without boundary. A Lipschitz domain $D$ in $M^{n}$ is a connected open set in $M^{n}$, whose closure is compact and whose boundary is Lipschitz.\footnote{$^1$Throughout this paper, we regard $M^{n}$ as a Riemannian manifold without loss of generality, since (by John Nash) $M^{n}$ can be equipped with a Riemannian metric by embedding it into an Euclidean space}.

A \emph{quasi-Lipschitz domain} in $M^{n}$ is defined by gluing some of the boundary points of a Lipschitz domain in $M^{n}$  in a prescribed manner (see Definition~\ref{D2.5} for a precise definition).

\begin{defn} \label{D1.1}
Given a family
\begin{equation} \label{e1.01}
  \mathcal{D} := \{ D(t) \subset M^{n} \mid 0 \leq t \leq b \}
\end{equation}
of quasi-Lipschitz domains $D(t)$ in $M^{n}$, such that
$D(t) \subsetneq D(r)$ for $0 \leq t < r \leq b$, with
\begin{align}\label{e1.02}
  D(t) = \bigcup_{s < t} D(s) \quad \textrm{and} \quad
  \overline{D(t)} = \bigcap_{r > t} \overline{D(r)},
\end{align}
we call $\mathcal{D}$ a \emph{$C^{0}$-deformation of enlarging quasi-Lipschitz domains in $M^{n}$}, or simply a \emph{monotone $C^{0}$-deformation} on $M^{n}$. For convenience, we may also simply call $\mathcal{D}$ a \emph{$C^{0}$-deformation} on $M^{n}$.
\end{defn}

The notion of ``quasi-Lipschitz domain" is introduced (in Definitions~\ref{D2.4} and \ref{D2.6}) for allowing topological changes of $D(t)$ along a given $C^0$-deformation in $t$. One of the purposes of this paper is to analyze the relevant Sobolev-theoretic properties near the ``singularities" created  by the gluing points on its boundary $\partial D(t)$.\\

For a quasi-Lipschitz domain $D$ in $M^{n}$, we consider $L^{2}(D)$ with the $L^{2}$-metric
\begin{align}\label{e1.03}
  \langle f, g \rangle_{L^{2}(D)}
  := \int_{D} fg \, dM
  = \int_{D} fg,
\end{align}
where $dM$ is the volume form of $M^{n}$ given by the equipped Riemannian metric on $M^{n}$. However, the notation ``$dM$" is usually omitted later, as seen in the last equality. Let
\begin{align}\label{e1.04}
  \mathcal{F}(D)
  := \{ u \in C^{\infty}(D) \cap C^{1}(\overline{D}) \mid
    u|_{\partial D} = 0 \}
\end{align}
and let $L$ be a globally defined strongly elliptic second-order operator on $M^{n}$ such that, for every quasi-Lipschitz domain $D \subset M^n$,
\begin{equation}
\label{e1.05}
  L \colon \mathcal{F}(D) \to C^{\infty}(D)
\end{equation}
is formally self-adjoint in $L^{2}(D)$, i.e.,
$\langle Lu,v \rangle _{L^{2}(D)} =\langle u,Lv \rangle _{L^{2}(D)}$. Locally, $L$ can be written as
\begin{equation} \label{e1.06}
  Lu = -a^{ij} \partial_{i} \partial_{j}u + b^{k} \partial_{k}u + cu,
\end{equation}
in which $\partial_{i} = \frac{\partial u}{\partial x^{i}}$; $a^{ij} = a^{ij}(x)$, $b^{k} = b^{k}(x)$ and $c = c(x)$ are smooth in $x \in D$. Note that $a^{ij} = a^{ji}$, and $x = (x^{1},x^{2},\ldots,x^{n}) \in \mathbb{R}^{n}$ denotes a local coordinate of $D$ in  
$M^{n}$. By saying that $L$ is strongly elliptic, we mean that
\begin{equation} \label{e1.07}
  a^{ij} \xi_{i} \xi_{j} \geq \alpha |\xi|^{2},
    \quad \textrm{$\alpha$ a positive constant},
\end{equation}
for all $\xi = (\xi_{1},\xi_{2},\ldots,\xi_{n}) \in \mathbb{R}^{n}$. It is easy to verify that $L$ of the form \eqref{e1.06} is  formally self-adjoint in $L^{2}(D)$, if and only if
\begin{equation}\label{1.5a}
  Lu = -\partial_{i}(a^{ij} \partial_{j}u) + cu.
\end{equation}
Remark that the strong ellipticity \eqref{e1.07} is preserved under coordinate changes. The operator $L$ naturally induces a symmetric bilinear form $B_{L}$ on $\mathcal{F}(D)$ by
\begin{equation} \label{e1.08}
\begin{split}
  B_{L}(u,v)
  &= \int_{D} (-\partial_{i}(a^{ij} \partial_{j}u) + cu) \cdot v \\
  &= \int_{D} Lu \cdot v
  = \int_{D} u \cdot Lv.
\end{split}
\end{equation}

Consider the Sobolev space $W^{1,2}(D)$ consisting of all $L^{2}$-functions on $D$, whose first weak derivatives exist and belong to $L^{2}(D)$. Equip on the boundary $\partial{D}$ of $D$ the codimension-one Hausdorff measure induced by the volume form $dM$ of $D$.  Denote
\begin{equation} \label{e1.10}
\begin{split}
 \mathcal{F}_{0}(D)
  &:= \big\{ u \in C^{\infty}(D) \cap C^{0}(\overline{D}) ;\,\,
    u|_{\partial D} = 0 \big\}, \\
  \mathcal{F}_{cpt}(D)
  &:= \big\{ u \in C^{\infty}(D) ;\, \,\overline{\operatorname{supp}u} \subset D
    \big\},
\end{split}
\end{equation}
where $\overline{\operatorname{supp}u}$ means the closure of the support of $u$ in $D$. Evidently, $\mathcal{F}_{cpt}(D) \subset \mathcal{F}(D) \subset \mathcal{F}_{0}(D)$.

\begin{defn}\label{D1.2} Given a Lipschitz domain $D$ in $M^n$, a Sobolev function $g\in W^{1,2}(D)$ satisfies the \emph{boundary $L^2$-vanishing} condition, if there exist $f_k \in C^{\infty}(D) \cap C^{0}(\overline{D})$ such that $\int_{\partial{D}}f_k^{2} \rightarrow 0$ and $ \|g-f_k\|^{2}_{W^{1,2}(D)} \rightarrow 0$, as $k \rightarrow \infty$. We denote the totality of such Sobolev functions by $H(D)$, throughout the present paper.

\end{defn}
By Theorem~\ref{T3.2}, which will be proved later in Section~\ref{S3}, the closures of
\[
  \mathcal{F}_{cpt}(D), \quad
  \mathcal{F}(D), \,\,\textrm{and} \quad
  \mathcal{F}_{0}(D) \cap W^{1,2}(D) \quad 
\]
in $W^{1,2}(D)$ all coincide with the same space $H(D)$. Thus $H(D)$ is also the space of Sobolev functions in $W^{1,2}(D)$ with \emph{zero trace} on the boundary $\partial D$, which is defined as the closure of $\mathcal{F}_{0}(D) \cap W^{1,2}(D)$ and is denoted by $W^{1,2}_{0}(D)$ in the literature. Theorem~\ref{T3.2} asserts that $H(D) = W^{1,2}_{0}(D)$ = the closure of $\mathcal{F}_{cpt}(D)$ in $W^{1,2}(D)$.

\begin{rmk} \label{R1.1}
Let $D$ be the open interval $(0,1)$ in $\mathbb{R}^{1}$, $u \in \mathcal{F}_{0}(0,1)$ and $u(x) = \sqrt{x}$ around $x = 0$. Then $u \notin W^{1,2}(0,1)$ since $u'(x) = 1/(2\sqrt{x})$ around $x = 0$ and $u' \notin L^{2}(0,1)$. Hence, in general, $\mathcal{F}_{0}(D) \nsubseteq W^{1,2}(D)$.
\end{rmk}

Extend the symmetric bilinear form $B_{L}$ given in \eqref{e1.08} to the Sobolev space $H(D)$, still denoted by $B_{L} \colon H(D) \times H(D) \to \mathbb{R}^{1}$, where $\partial_{i}u$ and $\partial_{i}v$ in \eqref{e1.08} are replaced by the weak derivatives $D_{i}u$ and $D_{i}v$ respectively, i.e.,
\begin{equation} \label{e1.11}
  B_{L}(u,v)
  = \int_{D} a^{ij} D_{j}u D_{i}v + cuv.
\end{equation}
We call $u \in H(D)$ a \emph{Jacobi field} on $D$ if $B_{L}(u,v) = 0$ for all $v \in H(D)$. For such Sobolev Jacobi fields, the classical equation ``$Lu = 0$" in the $C^\infty$-sense is 
generally meaningless, since second order weak derivatives of $u$ need not exist.

In the Sobolev setting on $H(D)$, we have the spectra of eigenvalues $\lambda_{k}$ of $B_L$:
\[
  \lambda_{1} \leq \lambda_{2} \leq \cdots \leq \lambda_{k} \leq \cdots
  \to \infty,
\]
obtained by the standard variational theory of Rayleigh quotients. Formally, the eigenfunctions $u_k$  satisfy $Lu_{k} = \lambda_{k} u_{k}$, where $u_{k}$ are the corresponding $k$-th eigenfunction of $L$ on $D$. Namely,
$B_{L}(u_{k},v)= \lambda_{k}\langle u_{k}, v\rangle _{L^2{(D)}}$, for any $v\in H(D)$, and
it is known that
\begin{equation} \label{e1.12}
\begin{split}
  \lambda_{k}
  &= \min \{ B_{L}(u,u) \mid u \in S_{k} \} \\
  &= \min_{V^{k}} \big\{ \max \{ B_{L}(u,u) \mid u \in V^{k} \cap S \} \big\},
\end{split}
\end{equation}
where $S$ is the unit sphere $\{ u \in H(D) \mid \|u\|_{L^{2}(D)} = 1 \}$ in $H(D)$,
\[
  S_{k}
  := \{ u \in S \mid \langle u, u_{j} \rangle = 0
    \textrm{ for all $j = 1,2,\ldots,k-1$.} \}
\]
and $V^{k} \subset H(D)$ denotes $k$-dimensional linear subspaces of $H(D)$. Furthermore, the eigenfunction $u_{k}$ attains the minimum of $B_{L}$ in the first formula of \eqref{e1.12}.

By elliptic regularity theory with Sobolev embedding theorem, the eigenfunctions $u_{k}$ are regular; in particular, $u_k \in \mathcal{F}_0(D)$. (A similar regularity argument was provided in Remark 2.8 in our previous paper \cite{H24}.) 

In particular, a Jacobi field on $D$ defined as above is also in $\mathcal{F}_{0}(D)$, by the same regularity argument. We may select $u_{k}$ such that $\{ u_{1},u_{2},\ldots \}$ constitutes an orthonormal basis of $H(D)$. Furthermore, by \eqref{e1.12}, it is straightforward to see that
\begin{equation} \label{e1.13}
  D \subset D' \quad \Longrightarrow \quad
  \lambda_{k}(D) \geq \lambda_{k}(D'), \quad \forall\, k = 1,2,\ldots,
\end{equation}
where $D$ and $D'$ are two domains in $M^{n}$ and $\lambda_{k}$ are the eigenvalues corresponding to the eigenfunction $u_{k}$.

\begin{defn} \label{D1.3}
We define the \emph{Morse index} $\operatorname{Ind}(D)$ on $D$ to be the maximal dimension of linear subspaces of $H(D)$ on which the bilinear form $B_{L}$ is negative definite. By the \emph{nullity} $\nu (D)$ of $L$ on $D$, we mean the dimension of all Jacobi fields on $D$, i.e., $\nu(D)$ is defined to be the number of linearly independent Jacobi fields on $D$.
\end{defn}

\begin{defn} \label{D1.4}
We say that the operator $L$ satisfies \emph{the unique continuation property on $M^{n}$} if the following holds: Given $D$, which is a domain (i.e. open and connected) in $M^{n}$, and $u \in C^{\infty}(D)$ with $Lu = 0$. If $u$ vanishes on a nonempty open subset of $D$, then $u \equiv 0$ on $D$.
\end{defn}

\begin{thmm}[Continuity and Existence] \label{TA}

Let $L$ be a self-adjoint strongly elliptic operator of second order defined on $M^{n}$, which satisfies the unique continuation property on $M^n$. For any monotone $C^{0}$-deformation $\mathcal{D}$ on $M^{n}$ (given in Definition~\textup{\ref{D1.1}}), we have: (i) the Sobolev continuity in $t$, i.e.,
\begin{equation} \label{e1.14}
  \overline{\bigcup_{s < t} H_{s}} = H_{t} = \bigcap_{r > t} H_{r},
\end{equation}
where $H_{t} := H(D(t))$ for any $t \in [0,b]$, and each $H_{t}$ is regarded as a subset of $H_{b}$, by letting the Sobolev functions in $H_{t}$ zero on $D(b)\, \setminus\, D(t)$; (ii) the eigenvalue continuity, i.e., the eigenvalues $\lambda_{k}(t):=\lambda_{k}(D(t))$ of $L$ depend continuously on $t$.
Furthermore, given any  quasi-Lipschitz domain $D$ of arbitrary topological type, and  a point $p_{0} \in D$, there exists a $C^{0}$-deformation $\mathcal{D}$ of enlarging quasi-Lipschitz domains $D(t)$ in $M^{n}$, $0 \leq t \leq b$, such that $D(0)$ is a small n-ball $B^{n}(p_{0})$ of $p_{0}$ in $M^{n}$ and $D(b)$ coincides with the prescribed domain $D$. 
\end{thmm}

Indeed, $B^{n}(p_{0})$ stated above is not necessarily an n-ball. It can be a \emph{ball-like neighborhood} of $p_0$, which means a diffeomorphic image of an n-ball. As a corollary of Theorem~\ref{TA}, given two Lipschitz domains $D$ and $D'$ in $M^{n}$, there exists a $C^{0}$-deformation $\{ D(t) \mid a \leq t \leq b \}$ from $D'$ to $D$ with eigenvalues of $L$ continuous in $t$ (by passing through a small $n$-ball with domains shrinking and then enlarging).

\begin{thmm} \label{TB}
(A global Morse index theorem)~Let $D \subset M^{n}$ be a Lipschitz domain in $M^{n}$ and $L$ be given as in Theorem~\textup{\ref{TA}}. The Morse index $\operatorname{Ind}(D)$ of $L$ on $D$ can be counted by
\begin{equation} \label{e1.15}
  \operatorname{Ind}(D) = \sum_{0 \leq t < b} \nu(D(t)),
\end{equation}
where $\mathcal{D} := \{ D(t) \subset M^{n} \mid 0 \leq t \leq b \}$ is a $C^{0}$-deformation of enlarging quasi-Lipschitz domains satisfying $D(0) = B^{n}(p_0)$ of a point $p_{0} \in D$ and $D(b) = D$. The topological types of quasi-Lipschitz domains $D(t)$ in $\mathcal{D}$ are allowed to change along the deformation.
\end{thmm}

Theorems~\ref{TA} and \ref{TB} will be proved in the later sections of this paper. The topological types of domains considered in the two theorems may change, for example, from a ball to a ring domain (letting $M^2$ a cylinder), or to a domain of various topological types. In this sense, we call Theorem~\ref{TB} a ``global" Morse index theorem, since  the deforming domain $D(t)$ can ``reach far", i.e. given \emph{any} Lipschitz domain $D$ in $M^n$, a small ball-like neighborhood $B^{2}(p_0)$ of any point $p_0$ in $D$ can be enlarged along the deformation to become  $D$. Note that the eigenvalues $\lambda_{k}(t)$ of $L$ are still continuous in $t$, and the Morse index $\operatorname{Ind}(D)$ can be counted by \eqref{e1.15}.\\

 In Smale's earlier contribution \cite{S65}, the Morse index theorem is ``local", as the initial domain $D(0)$ and the last one $D(b)$ must be of the same topological type. For example, if $M^2$ is a cylinder, a ring domain $D$ in $M^2$, can never be reached starting from a ball-like domain $D(0)\approx B^{2}(p_0)$ along a smooth deformation. Hereafter, ``$\approx$' denotes ``diffeomorphic to" throughout this paper.

%%=================================================================== sec. 2 ==
\section{Quasi-Lipschitz domains} \label{S2}
%%=============================================================================
\begin{defn} \label{D2.1}
A \emph{simple n-domain} is an open set in $\mathbb{R}^{n}$, diffeomorphic to an $n$-ball. A \emph{Lipschitz triple} $(U,\Gamma,V)$ in $\mathbb{R}^{n}$ is a simple n-domain $U$ in $\mathbb{R}^{n}$, attached with $\Gamma \subset \partial U$ and a simple (n-1)-domain $V$ in a hyperplane $H^{n-1}$ of $\mathbb{R}^{n}$, such that (see Figure~\ref{F2.1})
\begin{equation} \label{e2.1}
\begin{split}
  U
  &:= \{ (x,r) \in \mathbb{R}^{n} \mid u(x) < r < v(x), x \in V \subset H^{n-1}
    \}, \\
  \Gamma &:= \{ (x,u(x)) \in \mathbb{R}^{n} \mid x \in V \}, \\
  V
  &\subset H^{n-1}:
  = \{ (x,0) \in \mathbb{R}^{n} \mid x = (x^{1},x^{2},\ldots,x^{n-1}) \},
\end{split}
\end{equation}
where $(x,r)$ is a chosen coordinate of $U$ in $\mathbb{R}^{n}$, $u = u(x)$ is a Lipschitz function with Lipschitz constant $L_{0}$, i.e.,
\begin{equation} \label{e2.2}
  |u(x)-u(y)| \leq L_{0} \cdot |x-y|, \quad \forall\, x,y \in V,
\end{equation}
and $v(x)$ is a smooth function greater than $u(x)\,\forall\,x\, \in V$. We call $V$  the \emph{parameter set}.
\begin{figure}[!h]
\centering
\includegraphics[scale=0.5]{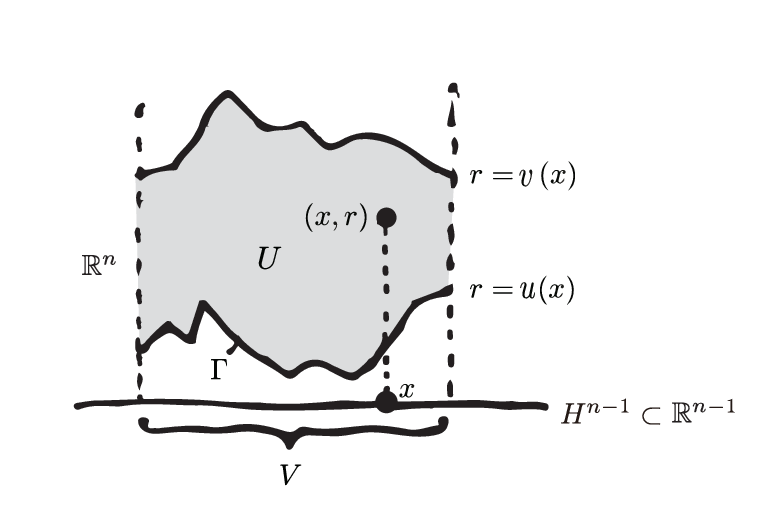}
\caption{Lipschitz triple in $\mathbb R^n$} \label{F2.1}
\end{figure}
\end{defn}
By \emph{a Lipschitz domain $D$ in $M^n$}, we mean that $D$ is a connected open set in $M^n$, and for each $p_0\in \partial D$ there is a coordinate neighborhood $(W,\psi)$ of $p_0$ in $M^n$, such that $(\psi(U),\psi(\Gamma),V))$ is a Lipschitz triple in $\mathbb{R}^{n}$, where $\psi:W\rightarrow \mathbb{R}^n$  is a coordinate map of $M^n$ around $p_0$,~ $U:=D\cap W,~ \Gamma:= \partial D\cap W$, and $V$ is a parameter set which is a simple (n-1)-domain in a hyperplane $H^{n-1}\subset\mathbb{R}^n$. We call $(U,\Gamma,V)$ \emph{a Lipschitz triple of domain $D$} in $M^n$.
 For $\Gamma'$ an open set of $\partial{D}$, $\Gamma'$ is said \emph{Lipschitz} if  every $p_0\in \Gamma'$ has a neighborhood $U$ which is in a Lipschitz triple of domain $D$ in $M^n$.

Identifying points in $W$ with $\psi(W)$, we may write $p=(x,r)$ for any $p\in U$. When $M^n$ is Riemannian, we can choose the coordinate neighborhood $(W,\psi)$ around $p\in M^n$, so that  each $r$-curve in $U$ of a Lipschitz triple $(U,\Gamma,V)$ is a geodesic in $M^n$.
This is possible by the standard arguments in elementary differential geometry. \\

 It is clear that every Lipschitz domain $D$ in $M^n$ has a finite family of Lipschitz triples, $(U_1,\Gamma_1, V_1), (U_2,\Gamma_2, V_2),\cdots,(U_h,\Gamma_h, V_h)$, with $\{\Gamma_j; j=1,2,\ldots,h \}$ covers $\partial D$.

\begin{defn} \label{D2.3}
Let $A$ and $B$ be two sets in $\mathbb{R}^{m}$ and $\mathbb{R}^{\ell}$ respectively, not both containing the origin $0$ of $\mathbb{R}^{m+\ell} := \mathbb{R}^{m} \times \mathbb{R}^{\ell}$, such that the normalization map $\nu \colon A \times B \to S^{m+\ell-1}(1)$, defined by $\nu(x,y) := \\
(x,y)/\! \sqrt{|x|^{2} + |y|^{2}}$ is injective. Here $S^{m+\ell-1}(1)$ is the unit sphere in $\mathbb{R}^{m+\ell}$. Let $m+\ell=n$. Define the cone with vertex at the origin $0$ of $\mathbb{R}^{m+\ell}$ and base $A \times B$ by
\begin{align}\label{e2.3}
  C_{0}(A \times B)
  := \{ t(x,y) \in \mathbb{R}^{m+\ell} \mid t \geq 0, \ x \in A, \ y \in B \}.
\end{align}

\end{defn} 
For example, we consider $C_{0}(S^{k} \times B^{n-k-1})$, where $B^n$ means the unit n-ball, $0 \notin A = S^k,\, 0 \in B=B^{n-k-1}$, and $S^{k} \subset \mathbb{R}^{k+1} = \mathbb{R}^{m},\, B^{n-k-1} \subset \mathbb{R}^{n-k-1}=\mathbb{R}^{\ell}$. Another useful example is 
$C_{0}(S^{k} \times S^{n-k-2})$, where $0\notin A= S^k \subset \mathbb{R}^{m}$, and $0 \notin B= S^{n-k-2} \subset \mathbb{R}^{\ell}$. Both have $m+\ell=n$. In particular, $C_{0}(S^{1} \times B^{1})$ and $C_{0}(S^{0} \times S^{1})$ in $\mathbb{R}^3$ will be considered later in Example~\ref{ExA} and Example~\ref{ExB} respectively.

\begin{defn} \label{D2.4}
A domain $G$ in $M^{n}$ is called \emph{diffeomorphic} under a map $\lambda$ into a cone $C_{0}(A \times B)$ \emph{relative to} $p \in G$, if (i)~$\lambda(p) = 0$, (ii)~$G$ is homeomorphic under $\lambda$ to a neighborhood $W$ of $0$ in $C_{0}(A \times B)$, (iii)~$G \setminus \{p\}$ is diffeomorphic to $W \setminus \{0\}$, and given any $(x,y) \in A \times B$, the curve $\lambda ^{-1} (t(x,y))$ is differentiable at $0$ with nonzero derivative, i.e.,

\begin{align}\label{e2.4}
  \hat{\nu} (x,y) :=\frac{d}{dt} \lambda^{-1}(t(x,y)) \Bigr|_{t=0} \neq 0.
\end{align}
\end{defn}

We are now ready to give  a precise definition of \emph{quasi-Lipschitz domains} in $M^{n}$. The construction of this definition is somewhat complicated. One may like to consult Example A given after Remark 2.4, while reading the definition.

\begin{figure}[!h]
\centering
\includegraphics[scale=0.5]{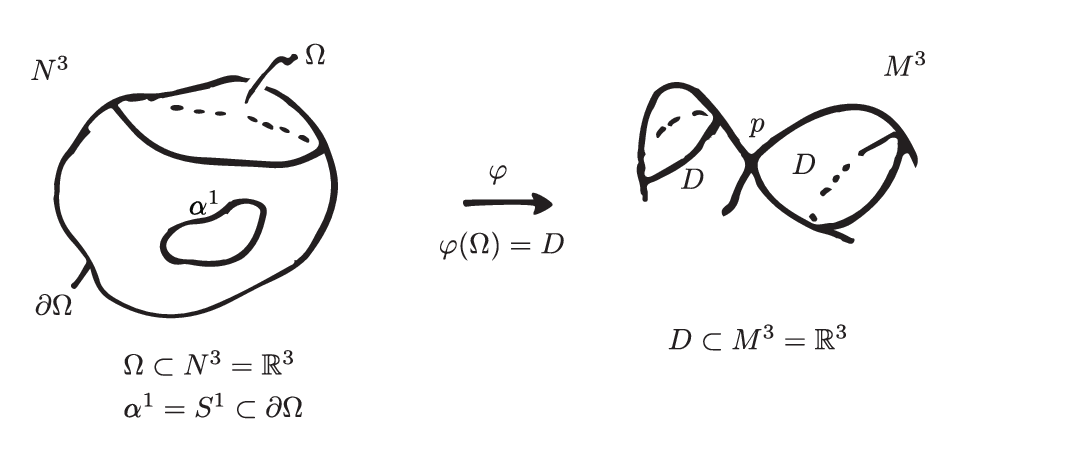}
\caption{gluing map ($n=3$)} \label{F2.2}
\end{figure}

\begin{defn} \label{D2.5}
Given a domain $D$ in $M^{n}$, with its compact closure $\overline{D}$ contained in $M^n$, i.e. $D\subset \subset M^n$. Consider a map
\begin{align}\label{e2.5}
  \varphi \colon \Omega \to D,\quad\text{(see Figure~\ref{F2.2})}
\end{align}
where $\varphi \in C^{\infty}(\Omega,M^{n}) \cap C^{0}(\overline{\Omega},M^{n})$, $\Omega$ is a domain in a smooth manifold $N^{n}$ with smooth boundary $\partial \Omega$, and $\Omega \subset \subset N^n$. If the following conditions are satisfied, $\varphi$ is called a \emph{gluing map} of $D$, and $D$ a \emph{quasi-Lipschitz domain} in $M^{n}$:
\begin{enumerate}
\item[(i)] $\varphi$ is a diffeomorphism between $\Omega$ and $D$. It can be extended continuously to a neighborhood of $\Omega$ in $N^n$, and $\partial D = \varphi(\partial \Omega)$.

\item[(ii)] There exist at most finitely many points $p_{1},\ldots,p_{\beta} \in \partial D$ (called \emph{glued points} of $\partial D$) such that each $\varphi^{-1}(p_{\ell})$ ($\ell = 1,\ldots,\beta$) is either (1) a $k$-dimensional smooth submanifold $\alpha^{k}$ of $\partial \Omega$ ( $0 < k < n-1$) which is connected and compact without boundary, or (2) a finite set $\alpha{^0}$ of points $q_{1},\ldots,q_{F}$ in $\partial \Omega$.  For example, $\alpha^{k}$ is a $k$-sphere $S^{k}$ lying in $\partial \Omega$ for $k\neq {0}$ (see Figures~\ref{F2.2},~\ref{F2.3}), and $\alpha^{0} = S^{0} = \{ q_{1}, q_{2} \}$. Note that $k$ depends on $p_{\ell}$. We call $\alpha^{k}$ the \emph{preglued set} of $p_{l}$.

\item[(iii)] $\Gamma'_{0} := \partial \Omega \setminus \varphi^{-1}(\{ p_{1},\ldots,p_{\beta} \})$ and $\Gamma_{0} := \partial D \setminus \{ p_{1},\ldots,p_{\beta} \}$ are homeomorphic under $\varphi$. Furthermore, $\Gamma_{0}$ is Lipschitz in the sense that $\forall p \in \Gamma_0$, there is a ``Lipschitz triple $(U,\Gamma, V)$ of the domain $D$" in $M^n$ (previously defined after Definition~\ref{D2.1}), where $p\in \Gamma = \Gamma_0$ locally.

\item[(iv)] Around a glued point $p \in \partial D$,the domain $D$ admits a \emph{sector domain} $\Lambda$ (see the following  \eqref{e2.6} and Figure~\ref{F2.05}). More precisely, letting $S^{k}$ and $B^{m}$ denote the k-sphere and m-ball respectively, there exists a small neighborhood $\widetilde{U}_{p}$ of $p$ in $M^{n}$ such that both the two maps
\begin{equation} \label{e2.6}
\begin{matrix}
  \lambda \colon & \Lambda~ := \widetilde{U}_{p} \cap D & \to
    & C_{0}(\alpha^{k} \times B^{n-1-k}), \\
  \lambda|_{\Lambda'} \colon
    & \Lambda' := \widetilde{U}_{p} \cap \partial D & \to
    & C_{0}(\alpha^{k} \times S^{n-2-k}),
\end{matrix}
\end{equation}
are diffeomorphic into the cones relative to $p$ with $\lambda (p) = 0$, in the sense of Definition~\ref{D2.4}, and given $x \in \alpha^{k}$, $y \neq y' \in S^{n-2-k}$,
\begin{align}\label{e2.7}
  \widehat{\nu}(x,y) \neq \widehat{\nu}(x,y'),
\end{align}
 where $\widehat{\nu}(x,y)$ is the unit vector tangent to the ray $\lambda^{-1}(t(x,y))$ at $t = 0$ with $t \geq 0$, $x \in \alpha^{k}$ and $y \in S^{n-2-k}$. The domain $\Lambda$ is called a \emph{sector domain} of $D$ around the glued point $p$. See Figures~\ref{F2.3}.
\end{enumerate}
\end{defn}
\begin{figure}[!h]
\centering
\includegraphics[scale=0.5]{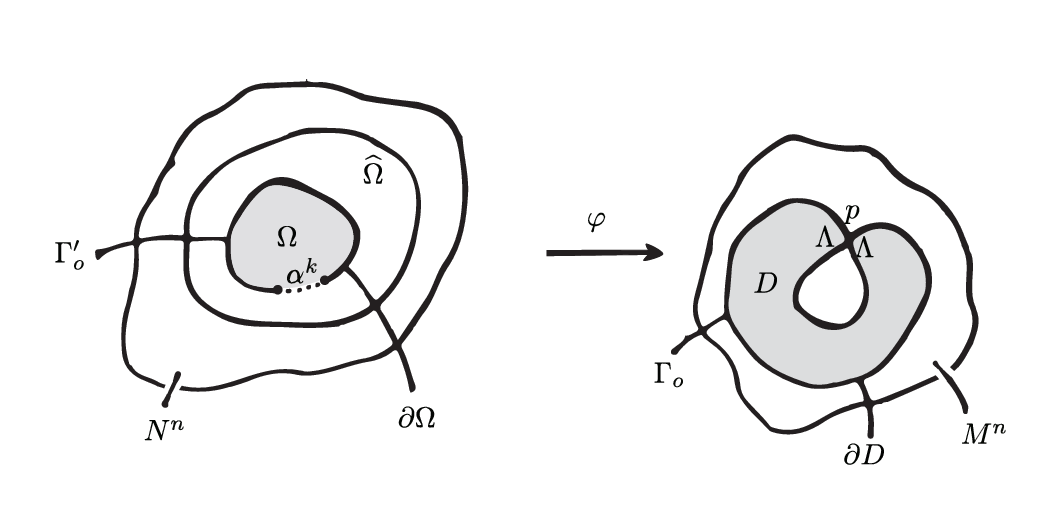}
\caption{a schematic illustration of a quasi-Lipschitz domain} \label{F2.3}
\end{figure}
\begin{defn} \label{D2.6}
For the gluing map $\varphi \colon \Omega \to D$ in Definition~\ref{D2.5}, we say that $\Omega$ is the \emph{model domain} of $D$ and $N^{n}$ is the \emph{model space}. The points $p_{1},\ldots,p_{\beta} \in \partial D$ in Definition~\ref{D2.5}(ii) are called the \emph{glued points} of $\partial D$.
\end{defn}

\begin{defn} \label{D2.7}
A domain $D$ in $M^{n}$ is called a \emph{quasi-Lipschitz domain} in $M^{n}$ if $D = \varphi(\Omega)$ for some gluing map $\varphi \colon \Omega \to D$, where $\Omega$ is the model domain in the model space $N^{n}$. The glued points of $\partial D$ are also called the \emph{singularities} of $\partial D$. See Figure~\ref{F2.3}.
\end{defn}

\begin{rmk} \label{R2.1}
We distinguish the quasi-Lipschitz domains in $M^{n}$ for the two cases of the preglued set $\alpha^{k}$ with $k = 0$ and $k \neq 0$ by the terms ``point-glued" and ``set-glued".
\end{rmk}

\begin{rmk} \label{R2.2}
We compare the terminologies in the two papers: In our previous paper \cite{H24}, we defined ``generalized Lipschitz domains" where the glued points of $\partial D$ are called ``joint points", and they are not necessarily finite and discrete as in this paper. However, no set-glued points are considered in \cite{H24}. In this paper, we focus especially on set-glued domains, in order to construct $C^{0}$-deformation of domains as required in the main theorems, i.e., Theorems~\ref{TA} and \ref{TB}.
\end{rmk}

\begin{defn}\label{D2.8}
(Properness) For a domain $D$ compactly contained in $M^{n}$, we say $D$ is proper, if $\partial D = \partial \overline{D} $.
\end{defn}

\begin{figure}[!h]
\centering
\includegraphics[scale=0.5]{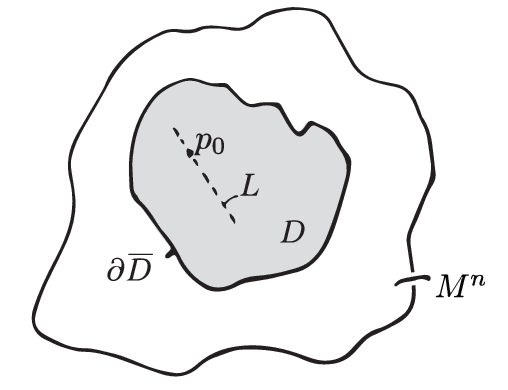}
\caption{$L\not\subset D$, $p_0\in\partial D\setminus\partial\bar{D}$; $D$ is not proper.} \label{F2.04}
\end{figure}

In Figure~\ref{F2.04}, the shaded area $D$ which is an open set of $M^n$ does not contain the line segment $L$ in its interior. Any point $p_0 \in L$ is included in $\partial{D}$. But $p_0 \notin \partial \overline {D}$.  Thus $p_0 \in \partial{D} \setminus \partial \overline{D}$, and $D$ is not proper.\\

\begin{prop} \label{p2.1}
 A quasi-Lipschitz domain is proper.
\begin{proof}
Step~1. It is easy to see that a domain $D$ in $M^n$ is proper, if and only if each boundary point $p \in \partial {D}$ is a limit point of the exterior $D_c := M^{n} \setminus \overline{D}$ of $D$. In fact, if there is $p_o \in \partial {D}$, such that $p_o$ has a neighborhood $N_{p_o}$ disjoint from $D_c$ , then $N_{p_o} \subset \overline {D}$, which means that $p_o \notin \partial \overline {D}$, and hence $\partial {D} \neq \partial \overline{D}$. Thus, D is not proper. On the other hand, if $D$ is proper, then every $p_o \in \partial{D}$ is in $\partial \overline{D}$, since $\partial D = \partial \overline{D} $. It means that $p_o$ is a limit point of $D_c$, and hence $D$ is proper.\\

Step~2. A Lipschitz domain $D$ is defined by having a Lipschitz triple around any $p \in \partial D$. Thus, every $p \in \partial D$ is clearly a limit point of $D_c$, and hence $D$ is proper. A quasi-Lipschitz domain $D$ is basically a Lipschitz domain with a set in $\partial D$ of dimension less than $n-1$ glued together in the manner defined in Definition~\ref{D2.5}. The gluing procedure maintains the property that every $p \in \partial D$ is a limit point of $D_c$. Therefore, a quasi-Lipschitz domain is still proper.
\end{proof}
\end{prop}

\begin{rmk} \label{R2.3}
Indeed, we can define quasi-Lipschitz domains in $M^{n}$, by allowing the preimage $\varphi^{-1}(p)$ of a glued point $p$ to consist of a finite union of disjoint $k_{i}$-dimensional smooth submanifolds of $\partial D$ with $0 < k_{i} < n-1$, instead of a single $\alpha^{k}$ as in Definition~\ref{D2.5}(ii). The main results of this paper carry over to this broader class of domains. However, to avoid the complication of language, we rather restrict $\varphi^{-1}(p)$ to be exactly one $\alpha^{k}$ when $k \neq 0$.
\end{rmk}

\begin{rmk} \label{R2.4}
In Figure~\ref{F2.05}(a) and Figure~\ref{F2.05}(b), the shaded area $D$ is a quasi-Lipschitz domain, but it is not the case in Figure~\ref{F2.05}(c), since $\Lambda$ in the figure is not a sector domain (see Definition~\ref{D2.5}(iv)).
\begin{figure}[!h]
\centering\includegraphics[scale=0.5]{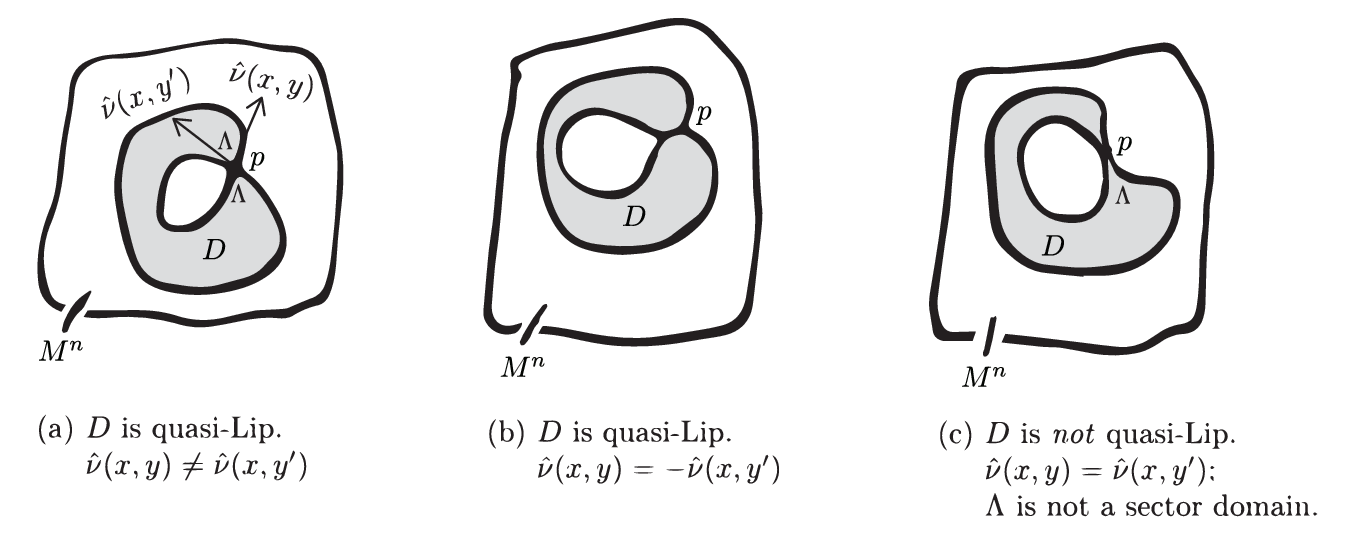}
\caption{examples of quasi-Lipschitz domains} \label{F2.05}
\end{figure}
\end{rmk}

We construct two typical examples to show the essential difference of set-glued and point-glued Lipschitz domains.

\begin{ex} \label{ExA}
Given $\Omega \approx D \approx$ a $3$-ball $B^{3}$. Let the preglued set $\alpha^{k} \approx S^{1} \subset \partial \Omega \approx S^{2}$, where $k = 1$, and the glued point $p \in \partial D \subset \overline{D} \subset M^{3} \approx \mathbb{R}^{3}$, with $\varphi^{-1}(p) = \alpha^{1}$. Here ``$\approx$" means ``diffeomorphic to". According to Definition~\ref{D2.5}, $\alpha^{1} \subset \partial \Omega$ \emph{cannot} be glued together by ``pinching" the circle $\alpha^{1} \approx S^{1}$ into a point $p$ in forming the domain $D$ as in Figure~\ref{F2.06}, since in that case $\Omega$ would not be homeomorphic to $D$ under the gluing map $\varphi$, violating Definition~\ref{D2.5}(i). In fact, the way to realize the quasi-Lipschitz domain $D$  is not to pinch the neck circle $\alpha^{1}$ along the ``interior" of $\Omega$. Instead, $\alpha^{1}$ should be shrunk and glued to a point $p$ along the ``exterior" of $\Omega$ in $N^{3}$.  
  There are various ways to see this kind of process illustrated in our $3$-space: One may  regard $\alpha^1$ as the edge of a dimple (as suggested by the referee of this paper), or regard it as the crater of a volcano.
\begin{figure}[!h]
\centering
\includegraphics[scale=0.6]{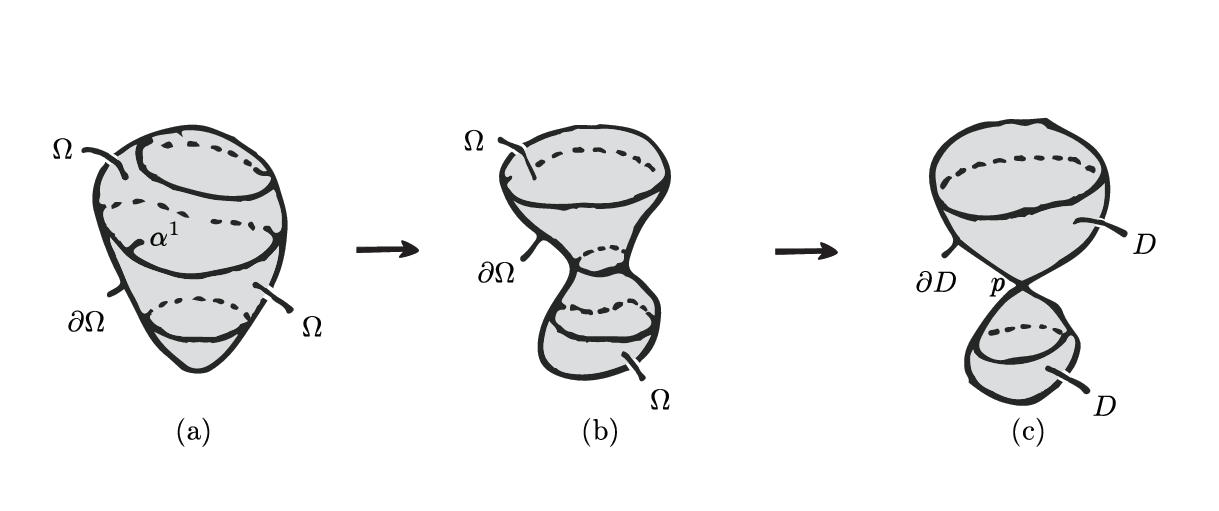}
\caption{Can't pinch preglued set $\alpha^1$ thru interior to a point} \label{F2.06}
\end{figure}

However, we try to visualize it employing ``inversion" which exchanges the interior of the $3$-ball $B^{3} \cong \Omega$ with its exterior, using the standard stereographic projection in the $4$-space. In other words, the domain $\Omega$ is now the exterior $B^{3}_{c}$ of the $3$-ball $B^{3}$ (subscript "c" means the “complement”), compactified with the point $p_{\infty}$ ``at infinity" (see the shaded area in Figure~\ref{F2.07}(a). Regard the preglued set $\alpha^{1}$ as a great circle $S^{1}$ of $S^{2} \cong \partial \Omega$ ($\approx \partial B^{3}_{c}$). The exterior $\Omega_{c}$ of $\Omega$ is now the inside ball $B^3$ bounded by $S^{2} \approx \partial \Omega$ (see the white area in Figure~\ref{F2.07}(a)). Pinch $\alpha^{1} = S^{1}$ into the origin $0$, which is now the glued point $p \in \partial D \subset \overline{D}$, if the model domain $\Omega$ is identified with the glued domain $D$ through $\varphi$, noting that $\Omega \approx D$ by Definition~\ref{D2.5}(i) (see the white area in Figures~\ref{F2.07}(b) and \ref{F2.07}(c)).
\begin{figure}[!h]
\centering
\includegraphics[scale=0.6]{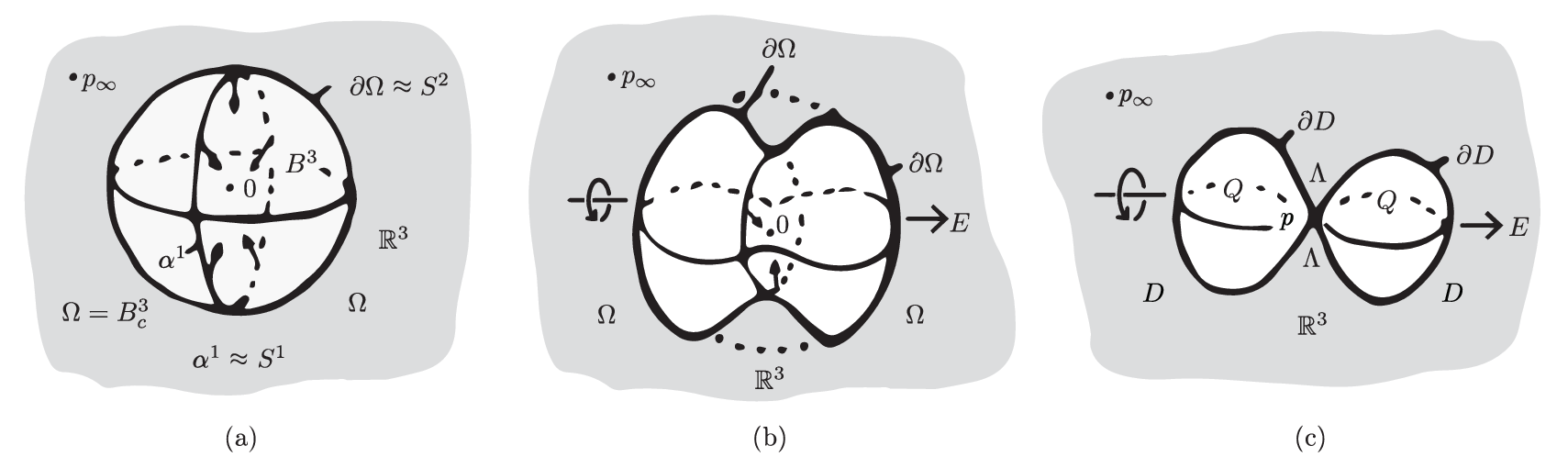}
\caption{Example A: set-glued Lipschitz domain $D$} \label{F2.07}
\end{figure}

The ``sector domain" $\Lambda$ of $D$ around the glued point $p$  has the shape of a domain like $C_{0}(\alpha^{1} \times B^{n-1-1}) = C_{0}(S^{1} \times I)$ (see \eqref{e2.6}), where $n-1-1 = 0$ and $B^{1} = I := (y, y')$, an interval with two end points $y$ and $y'$. It is obtained basically by rotating along $S^{1}$ the triangle $\sigma := C_{0}(x \times I) \subset \mathbb{R}^{2}$, $x \in \alpha^{1} \approx S^{1}$ (see Figure~\ref{F2.08}, where  $(y,y')$ is denoted by $(q_{1}, q_{2})$). Let $Q$ denote the domain obtained by rotating the planar area bounded by the digit "8" with the rotation axis $E$ in Figure~\ref{F2.07}(c). Then the required set-glued Lipschitz domain $D$ is $\mathbb{R}^{3}\cup\{p_{\infty}\}\setminus Q$, where the pregued set $\alpha^{k} = \alpha^{1} = S^{1}$ and the model domain $\Omega \approx$ $3$-ball $B^{3}$.

\begin{figure}[!h]
\centering\includegraphics[scale=0.5]{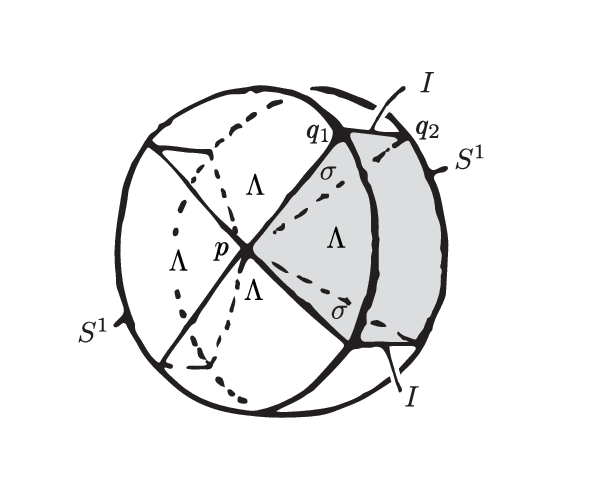}
\caption{sector domain $\Lambda\approx C_0(S^1\times I)$, obtained by rotating triangle $\sigma$ along $S^1$} \label{F2.08}
\end{figure}\end{ex}

\begin{ex} \label{ExB}
The second typical example is a point-glued Lipschitz domain $D$ with $k = 0$, $\alpha^{k} = \alpha^{0} = S^{0} =  \{ q'_{1}, q'_{2} \} \subset \partial \Omega$.  The required point-glued Lipschitz domain $D$ is given by pulling the two points $q'_{1},q'_{2} \in \partial \Omega$ in Figure~\ref{F2.09}(a) to the origin $0$ of the $3$-ball $B^{3}$, until the two points are glued together at $p \in \partial D$ . Here $B^{3}$ means the complement of~$\Omega$ in $N^{3} \approx \mathbb{R}^{3}$, by the inversion as in Example~\ref{ExA} (see Figures~\ref{F2.09}(a), \ref{F2.09}(b) and \ref{F2.09}(c)).
More precisely, if $\widetilde{Q}$ denotes the domain obtained by rotating the planar area bounded by the digit "8" with rotation axis $N$ in Figure~\ref{F2.09}(c), then $D$ is $\mathbb{R}^{3}\cup \{p_{\infty}\}\setminus{\widetilde{Q}}$. Here $k=0,n=3$, and the sector domain $\Lambda$ of $D$ around $p\in \partial D$  has the shape of a domain like $C_{0}(\alpha^{0} \times B^{n-1-0}) = C_{0}(S^{0} \times B^{2})$ (see \eqref{e2.6}), which consists of two circular 
cones joining their vertices at $p$. (See Figure~\ref{F2.09}(c).)

\begin{figure}[!h]
\centering
\includegraphics[scale=0.6]{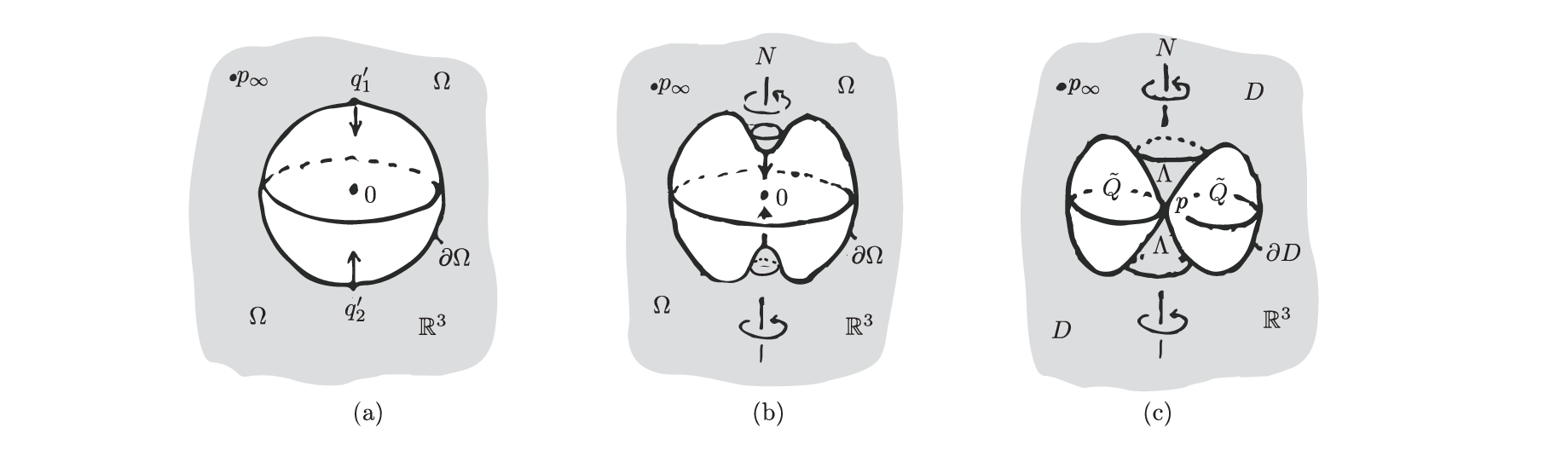}
\caption{Example~\ref{ExB}: point-glued Lipschitz domain $D$} \label{F2.09}
\end{figure}
\end{ex}

\begin{rmk} \label{R2.5}
By Proposition~\ref{p2.1}, any quasi-Lipschitz domain $D$ in $M^{n}$ is ``proper" , i.e., $\partial D = \partial \overline{D}$.  In Figure~\ref{F2.10}, the domain $D$ is not proper, since $p \in \partial D \setminus \partial \overline{D}$ is non-empty. The case of Figure~\ref{F2.10} may happen for a generalized Lipschitz domain (see \cite{H24}). However, for a quasi-Lipschitz domain $D \subset M^{n}$, the glued points are assumed finite and discrete in $\partial D$, and hence $D$ is always proper, as shown in Proposition~\ref{p2.1}.
\begin{figure}[!h]
\centering
\includegraphics[scale=0.5]{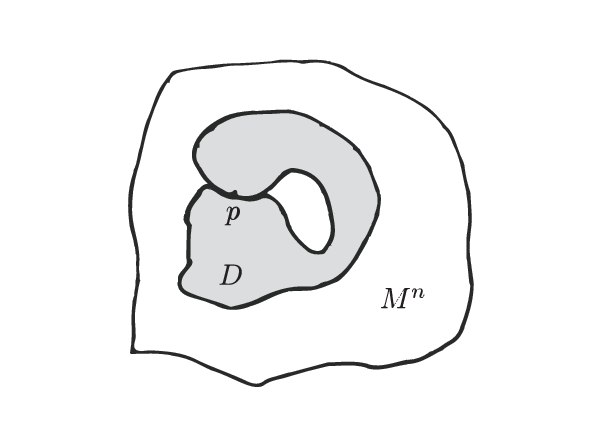}
\caption{non-properness; $D$ is not a quasi-Lipschitz domain} \label{F2.10}
\end{figure}
\end{rmk}

%%=================================================================== sec. 3 ==
\section {Singularities} \label{S3}
%%=============================================================================
Given a Sobolev function $g \in W^{1,2}(M^n)$.
By ``$g \equiv 0$ on a domain $G$ in $M^n$'', we mean that $g = 0$ almost everywhere on $G$, or equivalently, $\forall p \in G $, $\exists$ a neighborhood $N_p$ of $p$, such that $g=0$ almost everywhere on $N_p$. 

\begin{thm}[Trace theorem] \label{T3.1}
Given a quasi-Lipschitz domain $D \subset\subset M^{n}$ (i.e., the closure $\overline {D}\subset M^{n}$) and $g \in W^{1,2}(M^{n})$. If $g \equiv 0$ on $D_{c} := M^{n} \setminus \overline{D}$, then $g|_{D} \in H(D)$, and furthermore, $g|_{D}$ lies in the closure of $\mathcal{F}_{cpt}(D)$ in $W^{1,2}(D)$, where
\[
  \mathcal{F}_{cpt}(D)
  := \big\{ h \in C^{\infty}(D) \mid \overline{\operatorname{supp}h} \subset D
    \big\}
\]
(see also \eqref{e1.10}). In particular, $H(D)$ $= W^{1,2}_{0}(D)$, and the trace $Tg$ of $g$ vanishes on $\partial D$.
\end{thm}
Recall that a quasi-Lipschitz domain is proper. For a domain $D$ given as in Figure~\ref{F2.10} with $g\equiv 0$ on $D_c$, we note that $g$ may not be in the closure of $\mathcal{F}_{cpt}(D)$ in $W^{1,2}(D)$, since $D$ is not proper and neither a quasi-Lipschitz domain.

\begin{cor} \label{C3.1}
Given a quasi-Lipschitz domain $D \subset\subset M^{n}$, there exists no step function $g \in W^{1,2}(M^{n})$ with $g \equiv a$ in $D$ and $g \equiv b$ in $D^{c}$, $a \neq b$ being distinct constants.
\end{cor}
In fact, the corollary is a local observation. There is no $g \in W^{1,2}(M^{n})$ having a jump on an open subset of $\partial{D}$.\\

To establish the trace theorem, we now enter the jungle of local estimates, and then use the partition of unity to patch together the local results, while treating carefully the behavior of the Sobolev function $g$ around the ``singularities" (i.e., the glued points).

\begin{rmk} \label{R3.1}
The trace theorem (Theorem~\ref{T3.1}) is valid for generalized Lipschitz domains defined in \cite{H24}, where the glued points of $\partial D$ (called ``joint points" in \cite{H24}) are not necessarily finite and discrete.  Note that in the definition of ``generalized Lipschitz domains" in \cite{H24}, we assume the set-continuity which is equivalent to the properness in the present paper.  
Indeed, Theorem~\ref{T3.2}, Corollaries~\ref{C3.1} and \ref{C3.2} are also valid for generalized Lipschitz domains. The trace theorem developed here will simplify the tedious arguments in \cite{H24} to obtain the Sobolev continuity there.
However, the preglued set $\alpha^{k}$ with $k\neq{0}$ is not allowed for generalized Lipschitz domains in \cite{H24}, and we have to include this difficult case in the trace theorem (Theorem~\ref{T3.1}).
\end{rmk}

\begin{lem}[Local estimates]\label{L3.1}
 Given a Lipschitz triple $(U,\Gamma,V)$ in $M^{n}$, and a Sobolev function $g \in W^{1,2}(U)$. Let $f_{k} \in C^{\infty}(U) \cap C^{0}(\overline{U} \cap \Gamma)$ ($k = 1,2,3,\ldots$) with $\int_{\Gamma} f_{k}^{2} \to 0$ as $k \to \infty$, and
\begin{equation} \label{e3.01}
  \|f_{k}-g\|_{W^{1,2}(U)} \to 0 \quad \textrm{as $k \to \infty$}.
\end{equation}
Denote
\[
  C^{\infty}_{\Gamma}(U)
  := \big\{ h \in C^{\infty}(U) \mid \overline{\operatorname{supp}h}
    \cap \Gamma = \varnothing \big\}.
\]
Then there exist $h_{k} \in C^{\infty}_{\Gamma}(U)$ such that
\begin{equation} \label{e3.02}
  \|f_{k}-h_{k}\|^{2}_{W^{1,2}(U)} \to 0 \quad \textrm{as $k \to \infty$},
\end{equation}
and therefore
\[
  \|h_{k}-g\|^{2}_{W^{1,2}(U)} \to 0 \quad \textrm{as $k \to \infty$}.
\]
\end{lem}
\begin{proof}
Step~1. Use the notations for $(U,\Gamma,V)$ in \eqref{e2.1} and \eqref{e2.2}. Let $\overline{\delta}$ be a positive constant, define the $\overline{\delta}$-tubular neighborhood
\[
  N_{\overline{\delta}}
  := \{ (x,r) \in U \mid x \in V, u(x) < r < u(x)+\overline{\delta} \}.
\]
Put $\overline{\delta} = 4\delta$. Consider $\overline{u} := u+2\delta$ and the smooth approximation function $w(x)$ of $\overline{u}(x)$ given by (see Figure~\ref{F3.1})
\[
  w(x)
  := \int_{V} \overline{u}(y) \cdot \varphi_{\alpha}(y-x) \, dy,
    \quad \forall\, x \in V,
\]
where $\alpha > 0$ is a small number, and
\[
  \varphi_{\alpha}(x)
  := \frac{1}{\alpha^{n-1}} \varphi\left( \frac{x}{\alpha} \right)
\]
with $\varphi \in C^{\infty}(H^{n-1})$ the standard mollifier. (We may freely expand $U$ and $V$ in order to make $w(x)$ well-defined around $\partial V$ in $H^{n-1} \subset \mathbb{R}^{n-1}$.) For $x \in V$, we see that
\[
  |w(x)-\overline{u}(x)|
  = \left| \int_{V} (\overline{u}(y)-\overline{u}(x)) \varphi_{\alpha}(y-x)
    \, dy \right|
  \leq L_{0} \cdot \alpha
  = \delta,
\]
by choosing $\alpha = \delta/L_{0}$, where $L_{0}$ is the Lipschitz constant for $u = u(x)$.
\begin{figure}[!h]
\centering
\includegraphics[scale=0.5]{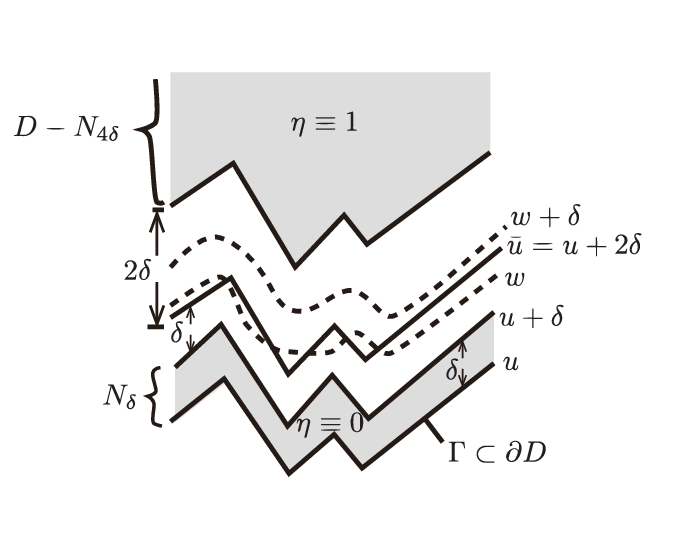}
\caption{local estimates} \label{F3.1}
\end{figure}

Step~2. Define $\eta \in C^{\infty}(U)$ by $\eta(x,r) := \eta_{0} \big( \frac{r-w(x)}{\delta} \big)$, where $\eta_{0} \in C^{\infty}(\mathbb{R})$ is a cut-off function with
\[
\begin{cases}
  \eta_{0}(s) = 0 &\textrm{for $s \leq 0$}, \\
  0 \leq \eta_{0}(s) \leq 1 &\textrm{for $0 \leq s \leq 1$}, \\
  \eta_{0}(s) = 1 &\textrm{for $s > 1$}.
\end{cases}
\]
Clearly, $\eta = 0$ on $N_{\delta}$ and $\eta = 1$ on $U \setminus N_{\overline{\delta}}$, recalling $\overline{\delta} = 4\delta$. Define
\begin{equation} \label{e3.03}
  h_{k} := f_{k} \cdot \eta \in C^{\infty}_{\Gamma}(U).
\end{equation}
We start to estimate the Sobolev norm of $f_{k}-h_{k}$:
\begin{equation} \label{e3.04}
  \int_{U} |f_{k}-h_{k}|^{2}
  = \int_{N_{\overline{\delta}}} |f_{k} (1-\eta)|^{2}
  \leq \int_{N_{\overline{\delta}}} f_{k}^{2}
  \leq C( \int_{N_{\overline{\delta}}} |Df_{k}|^{2})
    \cdot \delta^{2},
\end{equation}
in which the last inequality will be shown later by \eqref{e3.07} and Step 5;
\begin{equation} \label{e3.05}
\begin{split}
  \int_{U} |Df_{k}-Dh_{k}|^{2}
  &= \int_{N_{\overline{\delta}}} |Df_{k}(1-\eta) - f_{k}(D\eta)|^{2} \\
  &\leq C( \int_{N_{\overline{\delta}}} |Df_{k}|^{2} + \int_{N_{\overline{\delta}}} f_{k}^{2} |D\eta|^{2}),
\end{split}
\end{equation}
where $C$ in the formulas denotes positive constants for \emph{any} form, independent of $\delta$, $x$, $r$ and $f_{k},h_{k},\ldots$, such as $C = 2C = 2 = L_{0} = 1+L_{0}^{2} = \cdots$, etc. Sometimes, we use $C_{1},C_{2},\ldots$, etc., when it is necessary to distinguish them. Let the last term of \eqref{e3.05} be denoted by
\begin{equation} \label{e3.06}
  W := C\int_{N_{\overline{\delta}}} f_{k}^{2} \cdot |D\eta|^{2},
\end{equation}
which will be estimated later.\\

Step~3. $|D\eta|^{2} \leq C_{1}/\delta^{2}$ is to be established later in \eqref{e3.13}. On the other hand, we shall show later in Step 5 that
\begin{equation} \label{e3.07}
\begin{split}
  \int_{N_{\overline{\delta}}} f_{k}^{2}
  &\leq C( \int_{\Gamma} f_{k}^{2}) \cdot \delta
    + C( \int_{\Gamma} f_{k}^{2})^{1/2} \cdot P_{k}(\delta)^{1/2}
    \cdot \delta^{3/2}
    + C \cdot P_{k}(\delta) \cdot \frac{\delta^{2}}{2} \\
  &\leq C \cdot Q_{k}(\delta) \cdot \delta
    + C \cdot P_{k}(\delta) \cdot \delta^{2},
\end{split}
\end{equation}
where
\begin{align}
\label{e3.07a}
  Q_{k}(\delta)
  &:= \int_{\Gamma} f_{k}^{2} + ( \int_{\Gamma} f_{k}^{2} )^{1/2}
    \cdot P_{k}(\delta)^{1/2} \cdot \delta^{1/2}, \tag{3.7a}\\
\label{e3.07b}
  P_{k}(\delta)
  &:= \int_{N_{\overline{\delta}}} |Df_{k}|^{2}
  \leq \int_{N_{\overline{\delta}}} |Dg|^{2} + \alpha_{k}\tag{3.7b}
\end{align}
and $\alpha_{k} \to 0$ as $k \to \infty$ by \eqref{e3.01}.
Then we shall have
\begin{equation} \label{e3.08}
  W
  \leq C( \int_{N_{\overline{\delta}}} f_{k}^{2})
    \cdot \frac{C_{1}}{\delta^{2}}
  \leq C \cdot Q_{k}(\delta) \cdot \frac{1}{\delta} + CP_{k}(\delta).
\end{equation}
Combining \eqref{e3.04}, \eqref{e3.05}, \eqref{e3.06} and \eqref{e3.08}, we have
\begin{equation} \label{e3.09}
  \|f_{k}-h_{k}\|^{2}_{W^{1,2}(U)}
  \leq C_{1} \cdot Q_{k}(\delta) \cdot \frac{1}{\delta} + C_{2} P_{k}(\delta),
\end{equation}
by choosing a small $\delta = \delta_{0} < 1$.  This will prove \eqref{e3.02}, as required. More precisely, for any $\epsilon > 0$, choose $\delta = \delta_{0} < 1$ such that
\[
  \int_{N_{\overline{\delta}}} |Dg|^{2} < \frac{\epsilon}{4C_{2}}.
\]
Then selecting $k_{0}$ such that for $k > k_{0}$, $\alpha_{k} < \frac{\epsilon}{4C_{2}}$ is valid, where $\alpha_k$ is given in \eqref{e3.07b}, we obtain that $C_{2} P_{k}(\delta_{0}) < \frac{\epsilon}{2}$. By $\int_{\Gamma} f_{k}^{2} \to 0$ as $k \to \infty$, there exists $k_{1} > k_{0}$ such that for $k > k_{1}$,
\begin{equation} \label{e3.10}
  \frac{1}{\delta_0}Q_{k}(\delta_{0}) < \frac{\epsilon}{2C_{1}}.
\end{equation}
Hence for $\delta = \delta_{0}$, $k > k_{1}$, we obtain from \eqref{e3.09} that
\[
  \|f_{k}-h_{k}\|^{2}_{W^{1,2}(U)} < \epsilon.
\]
A \emph{key} point of the above estimation is that we \emph{squeeze} out $\delta^{2}$ of $\int_{N_{\overline{\delta}}} f_{k}^{2}$ to compensate the factor $1/\delta^{2}$ arising from the steepness of $|D\eta|^{2}$ as $\delta$ gets very small.\\

Step~4.  
We have
\begin{equation} \label{e3.11}
\begin{split}
  |D\eta|^{2}
  &\leq 2( |D_{x}\eta|^{2}
    + ( \frac{\partial \eta}{\partial r} )^{2}) \\
  &= 2[ \eta'_{0} ( \frac{r-w(x)}{\delta} )]^{2}
    \cdot [ ( \frac{-1}{\delta} |D_{x}w| )^{2}
    + ( \frac{1}{\delta} )^{2} ] \\
  &\leq 2~ \frac{C_{3}}{\delta^2}~ (1+|D_{x}w|^{2}),
\end{split}
\end{equation}
where $C_{3}$ is the bound of the term $|\eta'_{0}|^{2}$, a constant independent of $\delta$. In the first inequality of \eqref{e3.11}, the number~$2$ is attributed to the deviation of the Riemannian metric $g_{ij}$ of $M^{n}$ from the flat metric $(x,r)$ of $(U,\Gamma,V)$. Claim that
\begin{equation} \label{e3.12}
  |D_{x}w| \leq L_{0}.
\end{equation}
Given $x \in V$, for $z$ very close to $x$ in $V$, we have
\begin{align*}
  |w(z)-w(x)|
  &= \left| \int_{V} \overline{u}(y) \varphi_{\alpha}(y-z) \, dy
    - \int_{V} \overline{u}(y) \varphi_{\alpha}(y-x) \, dy \right| \\
  &= \left| \int_{B_{\alpha}(0)} \overline{u}(z+\zeta')
    \varphi_{\alpha}(\zeta') \, d\zeta'
    - \int_{B_{\alpha}(0)} \overline{u}(x+\zeta) \varphi_{\alpha}(\zeta)
    \, d\zeta \right| \\
  &\leq \int_{B_{\alpha}(0)}
    \big| \overline{u}(z+\zeta) - \overline{u}(x+\zeta) \big|~ \varphi_{\alpha}(\zeta) \, d\zeta
  \leq L_{0} \cdot |z-x|.
\end{align*}
Letting $z \to x$, we obtain \eqref{e3.12} and hence
\begin{equation} \label{e3.13}
  |D\eta|^{2}
  \leq 2C_{3}~ \frac{1+L_{0}^{2}}{\delta^{2}}
  = ~\frac{C}{\delta^{2}}
\end{equation}
by \eqref{e3.11}.\\

Step~5. Fix $k$, denote $f := f_{k}$ given in \eqref{e3.01}, in order to simplify the notations in the following computations. Claim \eqref{e3.07}: Note that $f = f(x,r)$, $x \in V$ and $u(x) < r < u(x)+\overline{\delta}$. Let $\overline{r} = r-u(x)$. But we still write $r$, replacing $\overline{r}$  for convenience. It is clear to see that
\begin{equation} \label{e3.14}
\begin{split}
  &\quad \int_{N_{\overline{\delta}}} f^{2} \, dr dx
  = \int_{V} ( \int_{0}^{\overline{\delta}} f(x,r)^{2} \, dr
    ) \, dx \\
  &= \int_{V} ( \int_{0}^{\overline{\delta}}
    [f(x,0) + \int_{0}^{r} f_{s}(x,s) \, ds]^{2}
    \, dr )^{2} \, dx \\
  &= \int_{V}( \int_{0}^{\overline{\delta}}
     [f(x,0)^{2} + 2f(x,0) \int_{0}^{r} f_{s}(x,s) \, ds
    + ( \int_{0}^{r} f_{s}(x,s) \, ds) ^{2}]
    \, dr \, dx,
\end{split}
\end{equation}
where $f_{s} := \frac{\partial f}{\partial s}(x,s)$. We simplify the computations by letting the area form = $dr \wedge dx$, disregarding the difference $\omega := dM - dr \wedge dx$.
Evidently, the simplification make sense, because we only  focus on the order of $\delta$, to which the deviation $\omega$ is negligible.

Denote the three terms in \eqref{e3.14} by $\alpha$, $\beta$ and $\gamma$, respectively. Then
\begin{gather*}
 |\alpha|
  = \int_{V} ( \int_{0}^{\overline{\delta}} (f(x,0))^{2} \, dr) \, dx
  = C( \int_{\Gamma} f^{2}) \cdot \overline{\delta}, \\
\begin{split}
  |\beta|
  &\leq 2\int_{V} |f(x,0)| ( \int_{0}^{\overline{\delta}}
    \left| \int_{0}^{r} f_{s}(x,s) \, ds \right| \, dr ) \, dx \\
  &\leq 2\int_{V} |f(x,0)| ( \int_{0}^{\overline{\delta}}
    ( \int_{0}^{r} f_{s}(x,s)^{2} \, ds)^{1/2}
   ( \int_{0}^{r} 1^{2} \, ds)^{1/2} \, dr ) \, dx \\
  &\leq 2\int_{V} |f(x,0)| ~( \int_{0}^{\overline{\delta}} |Df|^{2} \, ds
    )^{1/2}\cdot ~\frac{2}{3}~ \overline{\delta}^{3/2} \, dx \\
  &\leq \frac{4}{3}~ \overline{\delta}^{3/2}( \int_{V} f(x,0)^{2} \, dx
    \cdot \int_{V} \int_{0}^{\overline{\delta}} |Df|^{2} \, ds~ dx )^{1/2} \\
  &= C~\delta^{3/2} ( \int_{\Gamma} f^{2} )^{1/2}
    ( \int_{N_{\overline{\delta}}} |Df|^{2} )^{1/2}
  = C~\delta^{3/2} ( \int_{\Gamma} f^{2} )^{1/2} \cdot P_{k}(\delta)^{1/2},
\end{split} \\
\begin{split}
  |\gamma|
  &= \int_{V} \int_{0}^{\overline{\delta}} ( \int_{0}^{r}
    f_{s}(x,s) \, ds )^{2} \, dr dx
  \leq \int_{V} \int_{0}^{\overline{\delta}} ( \int_{0}^{r}
    |Df|^{2} \, ds )\, r \, dr dx \\
  &\leq \int_{V} ( \int_{0}^{\overline{\delta}} |Df|^{2} \, ds
    \cdot \int_{0}^{\overline{\delta}} r \, dr ) \, dx
=\, C\,\frac{\overline{\delta}^{2}}{2}~ P_{k}(\delta).
\end{split}
\end{gather*}
Thus we obtain \eqref{e3.07} and complete the proof of Lemma~\ref{L3.1}.
\end{proof}

We consider the simplest case of $n=1$ to review the essence of the above proof. The key point is in Step 3 of the proof that we estimated $\int_{N_\delta} f_{k}^2 = o(\delta^2)$ to compensates the order $c / \delta^2$ of $|D\eta|^2$. Write $f:=f_k$. Let $f(x)=x$, with D:=(0,1), then $\int_{N_\delta} f^2 =\int_{0}^{\delta} x^2 dx =  [x^{3}/3]_{0}^\delta = \delta^{3}/3$. Here, $f\in W^{1,2}(D)$, and the crucial term $W$ in \eqref{e3.06} satisfies
\begin{equation} \label{e3.15}
W:= C\int_{N_{\delta}}f^2 \,\cdot|D\eta|^2 = C \cdot \delta^{3}/3 \cdot C_{1}/\delta ^{2} = C_{2}\,\delta,
\end{equation}
which tends to zero, as $\delta$ tends to zero.

For a negative example, let $f(x)=\sqrt x$, $x\in(0,1)$, then $f$ is not a Sobolev function, since $f'(x)=1/\sqrt x$, and $\int_0^\delta\![f'(x)]^2\,dx=\infty$. (Recalling Remark~\ref{R1.1}).
Hence $f(x)$ does not meet \eqref{e3.01}, which requires in this case that $\int_0^\delta\!|f'(x)-Dg(x)|^2\,dx$ is arbitrarily small.

In order to figure out the intuitive meaning of the estimation that $\int_{N_\delta}\!f^2=o(\delta^2)$, one may expand $f(x)$ into
\[f(x)=f(0)+f'(0)x+E(x),\quad\text{with}~f(0)=0,\]
and $E(x)=o(x)$, to find that
\begin{align*}\int_0^\delta\![f(x)]^2\,dx&=\int_0^\delta\!(f'(0)x+E(x))^2\,dx\\&=[f'(0)]^2\int_0^\delta\!x^2\,dx+2f'(0)\int_0^\delta\!xE(x)\,dx+\int_0^\delta\![E(x)]^2\,dx\\&=O(\delta^3),\end{align*}
which is analogous to \eqref{e3.14} and \eqref{e3.07}.

\begin{thm}[Global version] \label{T3.2}
Given a quasi-Lipschitz domain $D$ in $M^{n}$, let $g$ be a Sobolev function in $H(D)$, i.e. $g$ satisfies the boundary $L^2$-vanishing condition defined by Definition~\ref{D1.2}. Then $g$ is in the closure of $\mathcal{F}_{cpt}(D)$ in $W^{1,2}(D)$ (see \eqref{e1.10}). In particular, the trace $Tg$ of $g$ on $\partial D$ vanishes, i.e. $H(D) = W^{1,2}_{0}(D)$.
 
\end{thm}

Recall some concepts: The hypothesis of $g\in H(D)$ means the existence of $f_{k} \in C^{\infty}(D) \cap C^{0}(\overline{D})$ ($k = 1,2,3,\ldots$) with $\int_{\partial D} f_{k}^{2} \to 0$ and $\|f_{k}-g\|_{W^{1,2}(D)} \to 0$ as $k \to \infty$. Given $g$ in $W^{1,2}(D)$, the \emph{trace} of $g$ on $\partial{D}$ is defined by the boundary value $u \in C^{0}(\partial{D})$, such that there exist $f_{k}\in C^{\infty}(D)\cap C^{0}(\overline{D})$ with $f_{k}|_{\partial{D}}=u$, and   $||f_{k}-g||^{2}_{W^{1,2}(D)}\rightarrow 0$, as $k\rightarrow \infty$. Particularly, $W^{1,2}_{0}(D)$  consists of functions in $W^{1,2}(D)$ with zero trace on $\partial {D}$. Namely, $W^{1,2}_{0}(D)$ is the closure of $\mathcal{F}_{0}(D) \cap W^{1,2}(D)$ in $W^{1,2}(D)$. \\

As a corollary of Theorem~\ref{T3.2}, we have the following closure theorem.

\begin{cor}[Closure theorem] \label{C3.2}
Given a quasi-Lipschitz domain $D$ in $M^{n}$,  
the Sobolev space $H(D)$ defined by Definition~\ref{D1.2} is identical with each of the closures of
\[
  \mathcal{F}_{cpt}(D), \quad
  \mathcal{F}(D), \,\,\textrm{and} \quad
  \mathcal{F}_{0}(D) \cap W^{1,2}(D) \quad 
\]
 in $W^{1,2}(D)$. 
\end{cor}

%%=================================================================== ==

%%=============================================================================
We shall complete the proofs of Theorem~\ref{T3.1} and \ref{T3.2} , by treating the behavior of a Sobolev function $g \in W^{1,2}(D)$ around the glued points of $\partial D$, which are the  singularities of the quasi-Lipschitz domains. The two theorems pave a way to establish the Sobolev continuity and eigenvalue continuity.\\

\renewcommand{\proofnamefont}{\bfseries}
\begin{proof}[Proof of the trace theorem (Theorem~\textup{\ref{T3.1}})]

Assume that the next Theorem~\ref{T3.2} is valid. For $g \in W^{1,2}(M^{n})$ with $g$ and $D$ given in Theorem~\ref{T3.1}, there exist $f_{k} \in C^{\infty}(M^{n})$, $k = 1,2,3,\ldots$, such that $\|f_{k}-g\|_{W^{1,2}(M^{n})} \to 0$ as $k \to \infty$. We claim that $\int_{\partial D} f_{k}^{2} \to 0$ as $k \to \infty$.\\

\emph{Step~$1$.} Given $p \in \partial D$ ($p$ may be a glued point of $\partial D$). We consider an $n$-ball $B_{p}$ in $M^{n}$, centered at $p$ and the set $G := B_{p} \cap D_{c}$, where $\partial G = \Gamma_{p} \cup K$, $\Gamma_{p} := \partial D \cap \partial G$, $K := \partial G \cap D_{c}$ (see Figure~\ref{F3.2}).
\begin{figure}[!htb]
\centering
\includegraphics[scale=0.5]{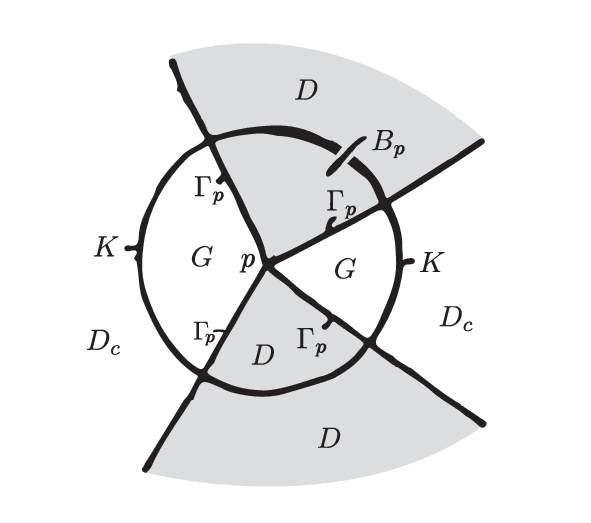}
\caption{to estimate $\int_{\partial D}f_k^2$} \label{F3.2}
\end{figure}

We use the method of moving frames to compute the relevant integrals. Let $\omega^{1},\omega^{2},\ldots,\omega^{n}$ be an orthonormal coframe in $G$, and $dA$ the area form of $\partial G$ (with slight smooth modification). The volume form $dM$ of $G$ is $dM = \omega^{1} \wedge \omega^{2} \wedge \cdots \wedge \omega^{n}$. Remark that
\begin{align*}
  dA
  &= a_{1}~ \widehat{\omega^1} \wedge \omega^{2} \wedge \cdots \wedge \omega^{n}
    + a_{2}~ \omega^{1} \wedge \widehat{\omega^2} \wedge \cdots \wedge \omega^{n}
    + \cdots \\
  &\quad + a_{n}~ \omega_{1} \wedge \omega^{2} \wedge \cdots \wedge \omega^{n-1}
    \wedge \widehat{\omega^n}, 
    \\
  df_{k}
  &= \partial_{i} f_{k}~ \omega^{i} \quad \textrm{(using summation convention)},
\end{align*}
where $``~\widehat{**}~"$ means~ ``absence".\\

\emph{Step~$2$.} We have
\begin{align*}
  \int_{\Gamma_{p}} f_{k}^{2} \, dA
  &\leq \int_{\Gamma_{p} \cup K} f_{k}^{2} \, dA
  = \int_{\partial G} f_{k}^{2} \, dA
  = \int_{G} d(f_{k}^{2} \, dA) \quad \textrm{(by Stoke's theorem)} \\
  &= \int_{G} df_{k}^{2} \wedge dA + \int_{G} f_{k}^{2} \, ddA
  = \int_{G} \big[ \partial_{1} (f_{k}^{2}) \omega^{1} + \cdots \big] \wedge dA
    + \int_{G} f_{k}^{2} \, ddA \\
  &= \int_{G} \big[ 2f_{k} (\partial_{1} f_{k}) a_{1} + \cdots \big] \omega^{1}
    \wedge \cdots \wedge \omega^{n}
    + \int_{G} f_{k}^{2} (\sum _{i} a_{ii}) \omega^{1} \wedge \cdots \wedge \omega^{n} \\
  &\leq C_1\left( \int_{G} f_{k}^{2} \right)^{1/2}
    \left( \int_{G} |Df_{k}|^{2} \right)^{1/2} \, dM
    + C_2\int_{G} f_{k}^{2} \, dM \\
  &\leq~ C\|f_{k}\|^{2}_{W^{1,2}(G)}
  \leq~ C\|f_{k}-g\|^{2}_{W^{1,2}(G)}
    \quad \textrm{(since $g \equiv 0$ in $G \subset D_{c}$)} \\
  &\to 0 \quad \textrm{as $k \to \infty$}.
\end{align*}
Applying Theorem~\ref{T3.2} to the given $g \in W^{1,2}(M^{n})$, we see that the trace $Tg$ of $g$ on $\partial D$ vanishes, and the trace theorem is established.
\end{proof}

\begin{proof}[Proof of Theorem~\textup{\ref{T3.2}}]
\emph{Step~$1$.} Let $g\in H(D)$ and $f_k\in C^\infty(D)\cap C^0(\overline D)$ be given in Definiton~\ref{D1.2}.
Around each glued point $p\in\partial D$, either point-glued or set-glued, we consider a
neighborhood $U_p$ of $p$ in $D$, and let
\[r=r(x):=\text{dist}_{M^n}(x,p),\, \text{with}~ r(p)=0;\]
\begin{equation}\label{e3.16}
U_p\supset\supset U_{4\delta}:=\{x\in D: 0<r(x)<4\delta\},
\end{equation}
where $\delta>0$ is a small number to be determined later.
Remark that the neighborhood $U_p$ could be of the shape like the Figure~\ref{F2.08} in Example~\ref{ExA}. Locally around $p$, we identify $U_p$ with an open neighborhood in
$C^0(\alpha^\ell\times B^{n-1-\ell})$ through the map $\lambda$ in \eqref{e2.6} of
Definition~\ref{D2.5}. Parametrize $x\in U_p$ by $x=(r,\xi)$, where $r>0$ and $\xi\in$ A:= an open set in $\alpha^\ell\times S^{n-1-\ell}$. Define on $U_p$ the cut-off function $\eta\in C^\infty(U_p)\cap C^0(\overline U_p)$ by
\begin{equation}\label{e3.17}
\eta(x):=
\begin{cases}0, &\text{for}~0\le r(x)\le\delta,\\1, &\text{for}~4\delta\le r(x),
\end{cases}
\end{equation}
where $\eta$ is increasing in $r$ and $|D\eta|<\frac{1}{6\delta}$ on $U_p$.
Given $k>0$, we define
\begin{equation}\label{e3.18}
\widehat h_k:=f_k\,\eta
\in C^\infty(U_p)\cap C^0(\overline U_p).
\end{equation}
Evidently, $\widehat h_k\equiv 0$ on $U_\delta:=\{x\in D:0<r(x)<\delta\}$, and $\overline{\text{supp}~\widehat h_k}\cap\Gamma_\delta=\varnothing$, where $\Gamma_\delta:=\overline U_\delta\cap\partial D=\{x\in\partial D:0\le r(x)\le\delta\}$.
\\

\emph{Step~$2$.}
Given $\varepsilon>0$, we claim that $\|f_k-\widehat h_k\|^2_{W^{1,2}(U_{4\delta})}<\frac{\varepsilon}{3\beta}$, $\beta$ being the number of glued points of $\partial D$ for $\delta=\text{some}~\delta_0>0$ and $k>\text{some}~k_0>0$ (see (iii) in Definition~\ref{D2.5}).
The proof is similar to the proof of Lemma~\ref{L3.1}, except now the $r$-curves are radiating from the glued point $p\in\partial D$, at which $r(p)=0$.

(i) To estimate $\|f_k-\widehat h_k\|^2_{W^{1,2}(U_{4\delta})}
<\frac{\varepsilon}{3\beta}$, we see that
\begin{equation}
\label{e3.19}
\int_{U_{4\delta}} |f_k-\widehat h_k|^2
\le \int_{U_{4\delta}} f_k^2
=\int_{U_{4\delta}} g^2+\alpha_k,
\end{equation}
and
\begin{align}
\label{e3.20}
\int_{U_{4\delta}} |Df_k-D\widehat h_k|^2
&\le C\left(\int_{U_{4\delta}} |Df_k|^2+W\right) \notag\\
&=C\left(\int_{U_{4\delta}} |Dg|^2+\beta_k\right)+CW,
\end{align}
where $\alpha_k\to0$, $\beta_k\to0$ as $k\to\infty$, and
\begin{align}\label{e3.21}
W:=C\int_{U_{4\delta}} f_k^2\,\,|D\eta|^2<\frac\varepsilon{9\beta}\quad\text{on}~U_{4\delta}.
\end{align}
Remark that \eqref{e3.19} and \eqref{e3.20} are based on the computations in \eqref{e3.04} and \eqref{e3.05}.

(ii) The equality in \eqref{e3.19} is elementary. It derives simply from Cauchy-Schwarz inequality.  However, we will compute it explicitly as follows:
\[
\int_{U_{4\delta}} f_k^2
=\int_{U_{4\delta}} g^2+(f_k^2-g^2)
=\int_{U_{4\delta}} g^2+\alpha_k,
\]
where
\begin{align}
|\alpha_k|
&=\left|\int_{U_{4\delta}}(f_k-g)(f_k+g)\right| \notag\\
&=\left|\int_{U_{4\delta}}(f_k-g)\bigl(2g+(f_k-g)\bigr)\right| \notag\\
&\le
2\left(\int_{U_{4\delta}} |f_k-g|^2
\int_{U_{4\delta}} g^2\right)^{1/2}
+\int_{U_{4\delta}} |f_k-g|^2 \notag\\
&=\left(\int_{U_{4\delta}} |f_k-g|^2\right)^{1/2}
\left[
2\left(\int_{U_{4\delta}} g^2\right)^{1/2}
+\left(\int_{U_{4\delta}} |f_k-g|^2\right)^{1/2}
\right]
\longrightarrow 0\label{e3.22}
\end{align}
as $k\to\infty$. Similarly, $\int_{U_{4\delta}} |Df_k|^2=\int_{U_{4\delta}}|Dg|^2+\beta_k$, with $\beta_k\to0$ as $k\to\infty$.
\\

(iii) It is not difficult to see that $|D\eta|^2\le C_1/\delta^2$,
which is similar to the proof of \eqref{e3.13}. Furthermore,
\begin{align}
W
&:=C\int_{U_{4\delta}} f_k^2\,\,|D\eta|^2
\le C\left(\int_{U_{4\delta}} f_k^2\right)\frac{C_1}{\delta^2} \notag\\
&=C\left[\left(\int_{U_{4\delta}} g^2\right)+\alpha_k\right]
\frac{C_1}{\delta^2} \notag\\
&=C\bigl(C_2\cdot\text{vol}(U_{4\delta})+\alpha_k\bigr)\frac{C_1}{\delta^2} \notag\\
&=C_3\delta^{n-2}+CC_1\alpha_k\cdot\frac{1}{\delta^2}.\label{e3.23}
\end{align}
Suppose $n:=\dim M^n>2$.
We will see \eqref{e3.21} immediately: For the given $\varepsilon>0$, we can find
$\delta_0>0$ such that $C_3\delta^{n-2}<\frac{\varepsilon}{18\beta}$, for $0<\delta\le\delta_0$. 
Then choose $k_0>0$ such that
\begin{equation}\label{e3.24}
CC_1\alpha_k\cdot\frac{1}{\delta_0^2}<\frac{\varepsilon}{18\beta},
\quad\text{for}\,\, k>k_0.
\end{equation}
Thus there exist $\delta_0>0$ and $k_0>0$ such that
\begin{equation}\label{e3.25}
W\le \frac{\varepsilon}{9\beta}
\quad\text{on }U_{4\delta_0}\,,\quad\text{for}\,\, k>k_0.
\end{equation}

(iv) Remark that both $\int_{U_{4\delta}} g^2$ and $\int_{U_{4\delta}} |Dg|^2$ are of the order $O(\delta^n)$. The argument in (iii) can be applied to the two corresponding terms in \eqref{e3.19} and \eqref{e3.20} to find the common $\delta_0>0$ and $k_0>0$, such that these terms in the right hand sides of \eqref{e3.19} and \eqref{e3.20} are less than $\frac\varepsilon{9\beta}$.
Now define globally a function $\widetilde f_k\in C^\infty(D)\cap C^0(\overline D)$ such that $\widetilde f_k:=\widehat h_k$ on each $U_p$, where
$p\in\partial D$ is a glued point, and $\widetilde f_k:=f_k$ elsewhere in $D$.
Then for $n:=\dim M^n>2$, we have
\begin{equation}
\label{e3.26}
\|f_k-\widetilde f_k\|^2_{W^{1,2}(D)}<\frac{\varepsilon}{3},
\end{equation}
and $\overline{\text{supp}~\widetilde f_k}\cap\Gamma_{\delta_0}=\varnothing$, around each glued point $p\in\partial D$.
\\

(v) As for $n=2$, the proof of \eqref{e3.26} is reduced to Lemma~\ref{L3.1}.
In fact, consider a neighborhood $U_p$ of a glued point $p\in\partial D$.
The boundary segment $\Gamma:=\overline U_p\cap\partial D=\Gamma_1\cup\cdots\cup\Gamma_F$
(see (2) of (ii) in Definition~\ref{D2.5}), where $\Gamma_i\cap\Gamma_j=\{p\}$ ,$\forall\,i\ne j$.
Furthermore, each $\Gamma_i$ is a simple Lipschitz line, which is
provided by requirement (v) of Definition~\ref{D2.5} concerning the sector
domains. Hence there are Lipschitz triples $(U_i,\Gamma_i,V_i)$ for
which Lemma~\ref{L3.1} can be applied to obtain $\widehat h_k$ approximating
$f_k$ in $W^{1,2}(U_i)$. Define a global
$\widetilde f_k\in C^\infty(D)\cap C^0(\overline D)$ as in (iv) of this
step, i.e., Step 2. We obtain \eqref{e3.26}.
\\

\emph{Step~$3$.} Consider $\Gamma'
:=\bigcup\{\Gamma_{\delta_0}: p \text{ is a glued point on }\partial D\}$, where $\delta_0>0$ is the common small number chosen for all glued points $p\in\partial D$, and $\Gamma_{\delta_0}:=\overline U_{\delta_0}\cap\partial D$.
Then $\Gamma'$ is a compact Lipschitz set contained in $\partial D$.
There exists a finite cover $\Gamma'_j$, $j=1,2,\ldots,\mu$, with the corresponding Lipschitz triples $(U_j',\Gamma_j',V_j')$.
Remark that $\int_{\partial D}\widetilde f_k^{\,2}\longrightarrow0$ as $k\to\infty$.
We can apply Lemma~\ref{L3.1} to these triples to obtain
$h'_j\in C^\infty(U'_j)\cap C^0(\overline U'_j)$ with $\overline{\text{supp}~h_j'}\cap\Gamma_j'=\varnothing$, such that $h_j'$ is sufficiently close to $\widetilde f_k$ in
$W^{1,2}(U'_j)$.
Use the standard partition of unity (see Section 3.4 in \cite{H24} for details) to patch together the local approximations  $\{h_j'~\text{on}~U'_j\,;\,j=1,2,\ldots,\mu\}$, and $\widetilde f_k$ on $G_0:=D\setminus\bigcup\{U'_j:j=1,2,\ldots,\mu\}$ to obtain $h_k\in C^\infty(D)\cap C^0(\overline D)$, such that $\text{supp}~h_k\subset\subset D$ and
\begin{equation}
\label{e3.27}
\|\widetilde f_k-h_k\|^2_{W^{1,2}(D)}<\frac{\varepsilon}{3}.
\end{equation}
Finally, we choose $k_0>0$, such that for $k>k_0$,
\begin{equation}\label{e3.28}
\|f_k-g\|^2_{W^{1,2}(D)}<\frac{\varepsilon}{3},
\end{equation}
by the assumption of Theorem~\ref{T3.2}.
Combining \eqref{e3.26}, \eqref{e3.27}, and \eqref{e3.28}, we
have
\[\|g-h_k\|^2_{W^{1,2}(D)}<\varepsilon,\]
where $h_k\in\mathcal F_{\mathrm{cpt}}(D)$ and $k$ sufficiently large. In other words, $g\in W^{1,2}_0(D)$.
This completes the proof.
\end{proof}

\renewcommand{\proofnamefont}{\itshape}

\begin{rmk}In the above Step 3 (of the proof of Theorem~\ref{T3.2}, we use the partition of unity to patch together the functions $\{h_j'~\text{on}~U_j'\,;\,j=1,2,\dots,\mu\}$ and $\{\tilde f_k~\text{on}~G_0\}$, where $G_0:=D\setminus(U_1'\cup U_2'\cup\cdots\cup U_\mu')$. Let us explain this construction in more detail as follows. Let the shrunk open set $U_j'$ be denoted by $G_j$ so that $\{G_0,G_1,\dots,G_\mu\}$ covers $D$. Let $\{\varphi_0,\varphi_1,\dots,\varphi_\mu\}$ be the partition of unity on $G_j$, $j=0,1,\dots,\mu$, i.e., 
\[1=\varphi_0+ \varphi_1+\cdots+\varphi_\mu,\]
where $\overline{supp\, \varphi_{j}} \subset G_j$. Write $\tilde f:=\tilde f_k$ on $G_0$. Consider the integrated function
\[h_k:=\varphi_0\tilde f+(\varphi_1h_1'+\varphi_2h_2'+\cdots+\varphi_\mu h_\mu').\]
Then
\begin{align*}\|\tilde f-h_k\|_{W^{1,2}(D)}^2&=\sum_{j=1}^\mu\int_{U_j'}\!\varphi_j^2(\tilde f-h_j')^2+\varphi_j^2(D\tilde f-Dh_j')^2+|D\varphi_j|^2(\tilde f-h_j')^2\\&\leq C\sum_{j=1}^\mu\left(\|\tilde f-h_j'\|_{W^{1,2}(D)}^2+\int_{U_j'}\!|D\varphi_j|^2(\tilde f-h_j')^2\right)\\&=O(\varepsilon)\end{align*}
for suitable $\delta_0>0$ with $k>k_0>0$, since $\varphi_j^2\leq1$ and $|D\varphi_j|^2\leq\frac C{\delta^2}$. The last inequality explains in detail how \eqref{e3.27} is obtained.

\end{rmk}

\section{Continuity theorems}

Given a quasi-Lipschitz domain $D\subset\subset M^n$, we have defined the space $H(D)$ as the space of Sobolev functions in $W^{1,2}(D)$, which satisfy the boundary $L^2$ vanishing condition.  In Corollary~\ref{C3.2}, we have proved the closure theorem, and established that $H(D)$ coincides with the closure of $\mathcal{F}_{cpt}(D)$ in $W^{1,2}(D)$.\\ 

When $W^{1,2}(M^n)$ is mentioned in the following, we assume that $M^n$ itself has compact boundary without loss of generality, since the concerned domains $D$ in $M^n$ have compact closures, and $M^n$ can be restricted to a smaller part while maintaining $D\subset\subset M^n$. 
  
\begin{prop}\label{p4.1}
 Let $g\in W^{1,2}(D)$, then  $g\in H(D)$ if and only if $g$ can be extended to $\tilde{g} \in W^{1,2}(M^n)$ with $\tilde{g}\equiv 0$ on $D_{c}:= M^n \setminus \overline{D}$.
\begin{proof}
 (i)\,To prove the “if part", we choose $f_k \in C^{\infty}(M^n)$ with  $ \| \tilde{g}-f_k\|^{2}_{W^{1,2}(M^n)} \rightarrow 0$, as $k \rightarrow \infty$.
Using the arguments in Step 1 of the proof of Theorem~\ref{T3.1}, we have $\int_{\partial{D}}f_k^{2} \rightarrow 0$, i.e. $g \in H(D)$, where $g$ is the restriction of $\tilde{g}$ into $D$. \\

(ii)\,On the other hand, suppose that $g \in H(D)$. We extend it to $\tilde {g} \in L^{2}(M^n)$ by letting $\tilde{g}\equiv 0$ on $D_{c}$. From Theorem~\ref{T3.1}, $g$ is in the closure of $\mathcal{F}_{cpt}(D)$ in $W^{1,2}(D)$, i.e.$\forall \varepsilon >0$ there exists $f\in \mathcal{F}_{cpt}(D)$ such that $\|f-g\|^{2}_{W^{1,2}(D)}< \varepsilon$. Extend $f$ to $\tilde{f} \in C^{\infty}(M^n)$ by letting $\tilde{f}\equiv 0$ in $D_c$. Then $\|\tilde{f}- \tilde{g}\|^{2}_{W^{1,2}(M^n)}< \varepsilon$. Thus $\tilde{g} \in W^{1,2}(M^n)$.
   
\end{proof}
\end{prop}

Based on this equivalence, we \emph{identify} $g$ with its extension $\tilde g $, for convenience. Namely, a Sobolev function of $H(D)$ is identified with the one of $W^{1,2}(M^n)$ with zero value on $M^n \setminus \overline{D}$.

\begin{thm}[Sobolev Continuity] \label{T4.1}
Given a monotone $C^{0}$-deformation
\[
  \mathcal{D} := \{ D(t) \subset M^{n} \mid 0 \leq t \leq b \}
\]
on $M^{n}$, defined by Definition~\textup{\ref{D1.1}}. Denote the Sobolev space $H_{t} := H(D(t))$ (Definition~\ref{D1.2}). Then $\mathcal{D}$ satisfies the Sobolev continuity
\begin{equation}\label{e4.1}\overline{\bigcup_{s<t} H_{s}} = H_{t} = \bigcap_{r>t} H_{r}.\end{equation}

\end{thm}

Note that the sobolev continuity \eqref{e4.1} is also valid for any of the three Sobolev spaces which are the closures of $\mathcal{F}_{cpt}(D(t)),
  \mathcal{F}(D(t))\,\textrm{and}\,
  \mathcal{F}_{0}(D(t)) \cap W^{1,2}(D(t))$ in $W^{1,2}(D(t))$, since they are all identical with $H(D(t))$.

\begin{proof} 
Step 1. Claim that $H_t=\bigcap\limits_{r>t}H_r$ for $t\in [0,b]$. The inclusion ``$\subset$'' is obvious, since $H_t\subset H_r$ for $r>t$.
We show the converse ``$\supset$''  as follows.

 \item[(i)] Given $g\in\bigcap\limits_{r>t}H_r$, we claim that $g\equiv0$ on $D(t)_c:=M^n\setminus\overline{D(t)}$, where $D(t)_c$ means the set of ``exterior points'' of $D(t)$ in $M^n$. Then by the equivalence of Proposition~\ref{p4.1}, we can identify $g\in H_{r}$ as
$g\in W^{1,2}(M^n)$ with $g\equiv 0$ on $D(r)_c:= M^n \setminus{\overline {D(r)}}$.
\item[(ii)]  For $V$ an open set in $D(t)_c$, with $\overline V\cap D(t)=\varnothing$, there exists $r'>t$ such that $\overline V\cap \overline{D(r')}=\varnothing$. Suppose not, there exists $r_n\searrow t$ and points $x_n\in \overline V\cap D(r_n)$.
Consider a convergent subsequence, still denoted by $\{x_n\}$, which tends to $x_0\in\overline{V}$.
Given any $r>t$, there exists $n_0$ such that $r>r_n>t$ for $n>n_0$.
Hence, $x_n\in D(r_n)\subset D(r)$ for $n>n_0$, and $x_0=\lim\limits_{n\to\infty}x_n\in \overline{D(r)}$.
Thus $x_0\in\bigcap\limits_{r>t}\overline{D(r)}=\overline{D(t)}$.
By \eqref{e1.02}, and $x_0\in\overline{V}\cap\overline{D(t)}$, against the given assumption of $\overline{V}\cap\overline{D(t)}=\varnothing$. 
\item[(iii)] Claim that $g\equiv0$ on $D(t)_c:M^n-\overline{D(t)}$, which means $g=0$ almost everywhere in $D(t)_c$.
For an open set $V$ with $\overline{V}\cap\overline{D(t)}=\varnothing$, there exists $r'>t$ such that $\overline V\cap D(r')=\varnothing$, by (i).
Since $g\in H_{r'}$, we have that $g\equiv0$ on $V$, for any open set $V\subset D(t)_c$. It means that $g\in W^{1,2}(M^n)$ with $g\equiv 0$ on $D(t)_c$.
Since $D(t)\subset\subset M^n$, we apply the trace theorem (Theorem~\ref{T3.1}) to $D(t)$, obtaining $g\in W_{0}^{1,2}(D(t)):=H_t$. Thus the right equality $H_t=\bigcap\limits_{r>t}H_r$ of \eqref{e4.1}
is proved.  \\

Step 2. (i) Now we shall show the left equality:  $\overline{\bigcup\limits_{s<t}H_s}=H_t$ for $t\in[0,b)$. This is much easier. To show ``$\subset$'': Given $g\in\overline{\bigcup\limits_{s<t}H_s}$. For any $\varepsilon>0$, there exists $f\in\bigcup\limits_{s<t}H_s$, such that $\|g-f\|_{H_b}^2\leq\varepsilon/2$. Let $f\in H_{s}$ for some $s<t$.

(ii) By the definition of $H(D)$, $f\in H_s = H(D(s))$ means that there exist $f_k \in C^{\infty}(D(s)) \cup C^{0} (\overline {D(s)})$, such that $\int_{\partial {D(s)}} f_{k}^2 \rightarrow 0$ and $\|f_k - f \|_{W^{1,2}(D(s))}^2 \rightarrow 0$, as $k \rightarrow 0$. To claim that $g\in H_t$, we have to show that there exist $u_k \in C^{\infty}(D(t)) \cup C^{0} (\overline {D(t)})$, such that $\int_{\partial {D(t)}} u_{k}^2 \rightarrow 0$ and $\|u_k - g \|_{W^{1,2}(D(t))}^2 \rightarrow 0$, as $k \rightarrow 0$. But if we let $u_k = f_k$ on $D(s)$, and let $u_k =0$ on $D(t) \setminus \overline D(s)$, $u_k$ is not the smooth extension of $f_k$ to $D(t)$. This is the difficulty of the problem. However, the trick is using Theorem~\ref{T4.2}, there exists $h \in \mathcal {F}_{cpt}(D(s))$ to approximate $f$ in $W^{1,2}(D(s))$.

 (iii) The smooth function  $h$ satisfies $\|f-h\|_{W^{1,2}(M^n)}^2<\varepsilon/2$. Furthermore, $h\in\mathcal{F}_{cpt}(D(s))\subset\mathcal F_{cpt}(D(t))$. Therefore $\|g-h\|_{W^{1,2}(M^n)}^2<\varepsilon/2+\varepsilon/2=\varepsilon$, and hence $g\in H_t$, as required.\\

Step 3. It remains to show $\overline{\bigcup\limits_{s<t}H_s}\supset H_t$.
Given $g\in H_t$.
For any $\varepsilon>0$, there exists $h\in\mathcal F_{cpt}(D(t))$ such that $\|g-h\|_{W^{1,2}(M^n)}^2<\varepsilon$.
Let $K:=\overline{\text{supp}~h}$.
For any $p\in K\subset D(t)$, there exists $s_p<t$ such that $p\in D(s_p)\subset D(t)$, since $D(t)=\bigcup\limits_{s<t}D(s)$ by \eqref{e1.02}.
Choose an open neighborhood $N_p$ of $p$ in $D(s_p)$. Evidently, $\bigcup\{N_p\,:\,p\in K\}\supset K$.
As $K$ is compact, there exist $p_1,\dots,p_r$ such that $\bigcup\{N_{p_i}\,:\,i=1,\dots,r\}\supset K$.
Hence $K\subset D(\overline s)$ where $\overline s:=\max\{s_{p_1},\dots,s_{p_r}\}$. Thus $h\in\mathcal F_{cpt}(D(\overline s))\subset H_{\overline s}$ with $\|g-h\|_{W^{1,2}(M^n)}^2<\varepsilon$.
Therefore $g\in\overline{\bigcup\limits_{s<t}H_s}$, and the proof of the Sobolev continuity \eqref{e4.1} is completed.
\end{proof}

\begin{thm}[Eigenvalue Continuity] \label{T4.2}
Given a monotone $C^{0}$-deformation
\[
  \mathcal{D} := \{ D(t) \subset M^{n} \mid 0 \leq t \leq b \}
\]
on $M^{n}$, defined by Definition~\textup{\ref{D1.1}}. Let $\lambda_{k}(t) := \lambda_{k}(D(t))$ be the eigenvalues of the given strongly elliptic operator $L$ (see \eqref{e1.05}, \eqref{e1.11} and \eqref{e1.13}). Then each $\lambda_{k}(t)$ is continuous in $t$. Furthermore, $\lambda_{k}(t)$ is strictly decreasing in $t$. 
\end{thm}

Frid--Thayer \cite{FT90} proved the eigenvalue continuity of the stability operator for monotone $C^{\infty}$-deformations of domains $D(t)$ in $M^{n}$. In that case, the domains $D(t)$ are diffeomorphic to each other, not allowing topological change along $t$. Owing to their works, the arguments of their proof can be carried over to our case of $C^{0}$-deformations of quasi-Lipschitz domains with a given strong elliptic operator $L$ defined on the domains. The reason is that the argument of their proof reduces essentially to the Sobolev continuity. For their smooth case, where the domains are diffeomorphic to each other, the Sobolev continuity is trivial. Nevertheless, it becomes nontrivial in our framework because of the topological change of the domains. The singularities occurring at glued points are also a non-negligible factor. We have to build up the celebrated closure theorem (Corollary~\ref{C3.2}) to support the Sobolev continuity, Theorem~\ref{T4.1}. These considerations explain why we must devote considerable effort in establishing Sobolev continuity before showing eigenvalue continuity.\\

In order to confirm the eigenvalue continuity for our case, we rewrite the proof of Frid-Thayer on the eigenvalue continuity here, in a formulation consistent with the framework of this paper.

\begin{proof}[Proof of Theorem~\ref{T4.2}]
\bigskip\noindent\emph{Step~$1$.} (Right-continuity) We claim the right
continuity of ${\lambda}_{k}(t)$. Given $r_{i} \searrow t$, i.e., $\lim_{i \to
\infty} r_{i} = t$, and $t < r_{i} < r_{i-1}$, $\forall\, i=1,2,3,\ldots$.  Denote by $\{u_{k}(t)\}$ an
orthonormal basis of eigenfunctions of ${L}_{t}:=
{L}(D(t))$ in $H_{t}$, corresponding to
$\{{\lambda}_{k}(t)\}$. For each $r_{i}$, let
\[
  v_{1}(r_{i}), v_{2}(r_{i}), \ldots, v_{k}(r_{i}), \ldots
  \in H_{r_{i}}
  \subset H_{b}
\]
be an orthonormal base of eigenfunctions of ${L}_{r_{i}}$ on $D(r_i)$, which 
corresponds to ${\lambda}_{j}(r_{i})$, $j = 1,2,\ldots,k,\ldots$. Recall that $v_{k}(r_{i}):=0$ on $D(b)\setminus D(r_i)$.
For each $j$, ${\lambda}_{j}(r_{i})$ is increasing in $i$ (See $\eqref{e1.13}$), and is bounded
by ${\lambda}_{j}(t)$. Clearly, ${\lambda}_{j}(r_{i}) \to$
some $\alpha_{j}\in \mathbb{R}$, as $i \to \infty$. We will show that $\alpha_{j} =
{\lambda}_{j}(t)$.

\bigskip\noindent\emph{Step~$2$.} Fix $j \in \mathbb{N}$.
The set $\{ v_{j}(r_{1}), v_{j}(r_{2}), \ldots \}$ is bounded in $L^{2}(D(b))$,
since $\|v_{j}(r_{i})\|_{L^{2}(D(b))} = 1$. Claim that it has a subsequence
converging to $v_{j}$ in $H_{t}$. We write $v(r_{i}):= v_{j}(r_{i})$ at
the moment by dropping the index $``j"$ to simplify notations. By \eqref{e1.07} and \eqref{e1.08}, we see that
\begin{align}
B_L(v(r_i),v(r_i))&=\int_{D(r_i)}\alpha^{k\ell}D_kv(r_i)D_\ell v(r_i)+cv(r_i)^2\notag\\
\label{e4.2}&\geq\overline\alpha\|Dv(r_i)\|_{L^2(D(b))}^2+\overline c,
\end{align}
where $\overline\alpha$ and $\overline c$ are the integrals of $\alpha=\alpha(x)$ and $c=c(x)$ over $D(r_i)$. On the other hand,
\begin{align}\label{e4.3}B_L(v(r_i),v(r_i))=\lambda_j(r_i)\|v(r_i)\|_{L^2(D(r_i))}^2=\lambda_j(r_i)\leq\lambda_j(t).\end{align}
Combining \eqref{e4.2}--\eqref{e4.3}, we see that $\|Dv(r_i)\|^2<c$, independent of $i$, and hence there is a subsequence of $v(r_i)$, which converges to a function $v_j$ in $L^2(D(b))$. Still denote the subsequence by $v_j(r_i)$, $i=1,2,3,\dots$.

Remark that $v(r_i)$ is regular on $D(r_i)$, i.e., $v(r_i)\in C^\infty(D(r_i))\cap C^0(\overline{D(r_i)})$. Let $w_{i\ell}:=v(r_i)-v(r_\ell)$, $i>\ell$, $D(r_i)\supset D(r_\ell)$. Computing $\|Dw_{i\ell}\|_{L^2(D(b))}^2$, and using the integration by parts, we see that
\begin{align}\int_{D(b)}\!|Dw_{i\ell}|^2&=\int_{D(b)}\!D_hw_{i\ell}D_hw_{i\ell}\notag\\
&=\int_{D(b)}\!D_h(w_{i\ell}D_hw_{i\ell})-\int_{D(b)}w_{i\ell}(D_h^2w_{i\ell})\to0\label{e4.4}\end{align}
as $i,\ell\to\infty$, since the former is the boundary term and the latter tends to zero due to $w_{i\ell}\to0$ in $L^2(D(b))$. Thus, $Dv(r_i)$ is a Cauchy sequence in $L^2(D(b))$.
Let $D_hv(r_i)\to\xi_h$ in $L^2(D(b))$. Then for each $\varphi\in\mathcal F_{cpt}(D(b))$,
\begin{align}\label{e4.5}\langle~\xi_h~,~\varphi~\rangle=\lim_{i\to\infty}\langle~D_hv(r_i)~,~\varphi~\rangle=\lim_{i\to\infty}\langle~v(r_i)~,D_h\varphi~\rangle=\langle~v_j~,~D_h\varphi~\rangle,\end{align}
which shows that $\xi_h=D_hv_j$ and $v_j\in H_b$.

\bigskip\noindent\emph{Step~$3$.} We will show that $v_j$ is the eigenfunction of $L_t:=L(D(t))$ on $D(t)$ with eigenvalue $\alpha_j$. Given $i\in\mathbb N$, $v(r_i)$, $v(r_{i+1})$, $\dots\in H_{r_i}:=H(D(r_i))$. The limit $v_j$ is in $H_{r_i}$ for any $r_i$.
By the Sobolev continuity, $v_j\in\bigcap\limits_{i=1}^\infty H_{r_i}=\bigcap\limits_{r>t}H_r=H_t$, noting that $H_s\subset H_r$ for $s\leq r$.
Let $H_t^2$ be the closure of $\mathcal F(D(t)):=\{h\in C^\infty((D(t))\cap C^0(\overline{D(t)})\,;\,h\Big|_{\partial D(t)}=0\}$ in $W^{2,2}(D(t))$. For any $w\in H_t^2$, we see that
\begin{align*}
\langle~v_j~,~L_tw~\rangle&=\lim_{i\to\infty}\langle~v(r_i)~,~L_tw~\rangle=\lim_{i\to\infty}B_L(v(r_i),w)\\&=\lim_{i\to\infty}\lambda_j(r_i)\langle~v(r_i)~,~w~\rangle=\alpha_j\langle~v_j~,~w~\rangle,\end{align*}
which means that $B_L(v_j,\varphi)=\alpha_j\langle v_j,\varphi\rangle$ for all $\varphi\in H_t$, since $H_t^2$ is dense in $H_t$. Therefore,
\begin{align}\label{e4.6}
L_tv_j=\alpha_jv_j,\end{align}
i.e., $\alpha_j$ is an eigenvalue of $L_t$ on $H_t$.

\bigskip\noindent\emph{Step~$4$.} We now claim that $\alpha_{j}$ is the
\emph{$j$-th} eigenvalue of ${L}_{t}$ on $H_{t}$. Clearly,
${\lambda}_{j}(t) \geq \lim_{i} {\lambda}_{j}(r_{i}) =
\alpha_{j}$. Suppose ${\lambda}_{1}(t) > \alpha_{1}$. As $\alpha_{1}$
is an eigenvalue of ${L}_{t}$, $\alpha_{1}$ must be some
${\lambda}_{h}(t)$ with $h > 1$. Thus we have
\[
  {\lambda}_{1}(t)
  \leq {\lambda}_{h}(t)
  = \alpha_{1}
  \lneqq {\lambda}_{1}(t),
\]
a contradiction. Hence ${\lambda}_{1}(t) = \alpha_{1}$. Inductively,
we see ${\lambda}_{j} = \alpha_{j}$. Thus
\[ %\tag{6.6} \label{e6.6}
  \lim_{r_{i} \searrow t} {\lambda}_{j}(r_i)
  = \alpha_{j}
  = {\lambda}_{j}(t),
\]
which shows the right-continuity.

\bigskip\noindent\emph{Step~$5$.} (Left-continuity) Let $\{u_{k}\}$ be an orthonormal basis
of eigenfunctions defined in Step~1. Given $k \in \mathbb{N}$. Consider $j =
1,2,\ldots,k$. By the first equality of Theorem~\ref{T4.1}
, i.e., $\overline{\bigcup_{s
< t} H_{s}} = H_{t}$, there exist $s_{1}, s_{2}, \ldots \nearrow t$ and
$\overline{v}_{j}(s_{i}) \in H_{s_{i}} \subset H_{t}$, $\forall\, i =
1,2,\ldots$, such that $\overline{v}_{j}(s_{i}) \to u_{j}$ in $H_{t} \subset
H_{b}$ as $i \to \infty$. Using min-max principle~\eqref{e1.12}, we see that
\begin{align*}
  {\lambda}_{k}(s_{i})
  &= \min_{V^{k} \subset H_{s_{i}}}
    \{ \max_{v \in V^{k} \cap S} B_{L}(v,v) \}
    \qquad\qquad \textrm{(where $\dim V^{k} = k$)} \\
  &\leq \max_{v \in \langle \overline{v}_{1}, \ldots, \overline{v}_{k} \rangle
    \cap S} \: B_{L}(v,v) \qquad\qquad\quad\;
    \textrm{(where $\overline{v}_{j} \equiv \overline{v}_{j}(s_{i})$)} \\
  &\xrightarrow{\textrm{as } i \to \infty}
    \max_{u \in \langle u_{1}, \ldots, u_{k} \rangle \cap S} \:
    B_{L}(u,u)
    = {\lambda}_{k}(t)
    \leq {\lambda}_{k}(s_{i}), \quad \forall\, i,
\end{align*}
noting that $S$ is the unit sphere in $H_{s_i}$ or in $H_{t}$. Therefore
\[ %\tag{6.7} \label{e6.7}
  \lim_{s_{i} \nearrow t} {\lambda}_{k}(s_{i})
  = {\lambda}_{k}(t).
\]
The left-continuity is proved.

\bigskip\noindent\emph{Step~$6$.} (Strictness) It remains to show
\[
  s < t \;\; \Rightarrow \;\;
  {\lambda}_{k}(s) \gneqq {\lambda}_{k}(t),
    \;\forall\, k \in \mathbb{N}.
\]
Suppose $\exists\, k > 0$ with ${\lambda}_{k}(s) =
{\lambda}_{k}(t) \equiv \lambda$ for $s < t \leq b$. We shall
construct $u \in H_{s} \subset H_{t}$ such that
\[ %\tag{6.5} \label{e6.5}
  {L}_{s}u = \lambda u \quad \textrm{and} \quad
  {L}_{t}u = \lambda u,
\]
which is against $D(s) \subsetneqq D(t)$. Let $u_{1}, u_{2}, \ldots$ and
$v_{1}, v_{2}, \ldots$ be orthonormal bases of eigenfunctions of
${L}_{s}$ on $H_{s}$ and of ${L}_{t}$ on $H_{t}$,
respectively. Choose $u = a_{1} u_{1} + \cdots + a_{k} u_{k} \in H_{s} \subset
H_{t}$, with $\|u\| = 1$, such that $\langle \: u,v_{1} \: \rangle = \cdots =
\langle \: u,v_{k-1} \: \rangle = 0$. Clearly, such $a_{i}$'s are solvable, and
$u \in C_{0}^{\infty}(D(s))$. Write
\[ %\tag{6.8} \label{e6.8}
  u = \langle \: u, v_{k} \: \rangle \: v_{k}
    + \langle \: u,v_{k+1} \: \rangle \: v_{k+1} + \cdots.
\]
Then
\begin{equation} \label{e4.7}
\begin{split}
  \langle \: {L}_{t}u, u \: \rangle
  &= {\lambda}_{k}(t) \: \langle \: u,v_{k} \: \rangle^{2}
    + {\lambda}_{k+1}(t) \: \langle \: u, v_{k+1} \: \rangle^{2}
    + \cdots \\
  &\geq {\lambda}_{k}(t) \:(\:\langle \: u,v_{k} \: \rangle^{2}
    + \langle \: u,v_{k+1} \: \rangle^{2} + \cdots \:)
  = {\lambda}_{k}(t)
  = \lambda.
\end{split}
\end{equation}
But
\begin{equation} \label{e4.8}
\begin{split}
  \langle \: {L}_{s}u, u \: \rangle
  &= {\lambda}_{1}(s) a_{1}^{2} + \cdots
    + {\lambda}_{k}(s) a_{k}^{2} \\
  &\leq {\lambda}_{k}(s) \:(\: a_{1}^{2} + \cdots + a_{k}^{2} \:)
  = {\lambda}_{k}(s)
  = \lambda.
\end{split}
\end{equation}
Given $w \in H_{t}$, we know
\[ %\tag{6.11} \label{e6.11}
  \langle \: {L}_{s}u, w \: \rangle
  = \langle \: {L}_{t}u, w \: \rangle - E
  \neq \langle \: {L}_{t}u, w \: \rangle,
\]
in general, although ${I}_{s}(u,w) = {I}_{t}(u,w)$. Here
$E$ is the boundary term $\int_{\partial D(s)} (D_{\nu} u) w$. By choosing $w =
u \in H_{s} \subset H_{t}$, we have $w \equiv 0$ on $\partial D(s)$ and hence
$E = 0$. Thus
\begin{equation} \label{e4.9}
  \langle \: {L}_{s}u, u \: \rangle
  = \langle \: {L}_{t}u, u \: \rangle.
\end{equation}
A combination of \eqref{e4.7}, \eqref{e4.8} and \eqref{e4.9} yields that
$\lambda \geq \langle \: {L}_{s}u, u \: \rangle = \langle \:
{L}_{t}u, u \: \rangle \geq \lambda$. It forces all the inequalities
in \eqref{e4.7} and \eqref{e4.8} to be equalities. Hence
${\lambda}_{1}(s) = {\lambda}_{2}(s) = \cdots =
{\lambda}_{k}(s) = \lambda$, and ${L}_{s}u = \lambda u$. On
the other hand, ${\lambda}_{l}(t) \to \infty$ as $l \to \infty$.
Looking at \eqref{e4.7}, there is an $m \geq 0$ such that $u = \langle \:
u,v_{k} \: \rangle \: v_{k} + \cdots + \langle \: u,v_{k+m} \: \rangle \:
v_{k+m}$, and each of $v_{k}, \ldots, v_{k+m}$ has the same eigenvalue
$\lambda$. Thus
\[{L}_{t}u = \lambda u.\]
Now $u \in H_{s} \subset H_{t}$, $u \equiv 0$ on $D(t) - D(s)$. By the
assumption that $D(t) - D(s) \neq \phi$, there is a nonempty open set $U$
contained in $D(t) - D(s)$. Using Hopf's sphere theorem, we see that $u \equiv
0$ on $D(t)$, contradicting $\|u\| = 1$.
\end{proof}

%%=================================================================== sec. 5 ==
\section{The existence}
%%=============================================================================

We are ready to show the existence of $C^0$-deformations allowing topological changes as stated in Theorem~\ref{TA}. Let the operator $L$ on $M^{n}$ and a quasi-Lipschitz domain $D$ in $M^{n}$ be given in Theorem~\ref{TA}. For any $p_{0} \in D$, we shall construct the required $C^{0}$- deformation on $M^{n}$, along which both Sobolev continuity and eigenvalue continuity of $L$ hold, such that $D(0)$ is a small $n$-ball $B^{n}(p_{0})$ and $D(b)$= the given $D$.\\

We first present two amusing examples in lower dimensions, before presenting the general construction of the $C^0$-deformation in the proof of Theorem~\ref{TA}.

\begin{ex}\label{ExC}
Consider an annular domain $D:= S^{1}\times I$ in the plane $M^2$ ($\approx \mathbb{R}^2$),
 where $I$:= the interval $[0,1]$. Its boundary consists of $\Gamma_{+}$ and $ \Gamma_{-}$ which are the outer and the inner simple curves respectively. See Figure~\ref{F5.1}(a). Let $B^{2}(p_0)$ be a small 2-ball $\subset D \subset M^2$ centered at $p_{0}\in D$ with radius $\delta$. 
 
 We shall construct a $C^0$-deformation with $D(0)=B^{2}(p_0)$ and $D(b)=D$. To visualize the construction, we pull $\Gamma_{+}$ upward along the positive $z$-axis and pull $\Gamma_{-}$  downward along the negative $z$-axis, as illustrated in Figure~\ref{F5.1}(b). The inner complementary domain $D_{c}^{-}$ becomes a bowl-shaped region at the bottom of  the cylinder, while the outer complementary domain $D_{c}^{+}$
lies above $D$. See Figure~\ref{F5.1}(b).
\begin{figure}[!htb]
\centering
\begin{minipage}[t]{0.6\textwidth}
\centering\includegraphics[scale=0.4]{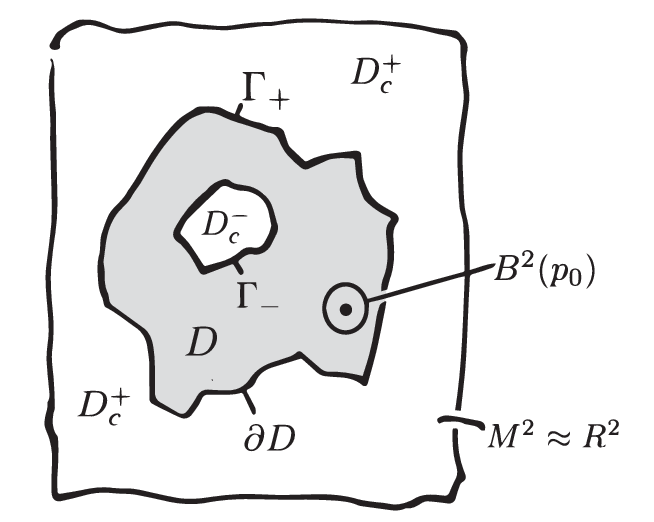}

(a) the given Lipschitz domain $D$\\\hspace{-5.4em}with one hole
\end{minipage}\hspace{-3em}
\begin{minipage}[t]{0.45\textwidth}\centering\includegraphics[scale=0.4]{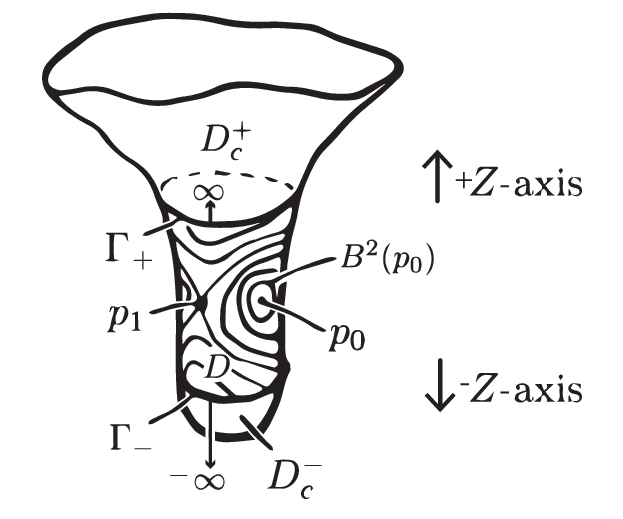}

(b) deforming $B^2(p_0)$ to\\~~~one-hole domain $D$
\end{minipage}
\caption{Example~\ref{ExC}} \label{F5.1} 
\end{figure}

Observe that $D$ now appears as a region like a vertical cylinder, where $D_{c}^{+}$ extends to infinity. Moreover, $M^{2}= D_{c}^{-} \cup \overline{D} \cup D_{c}^{+}$ is diffemorphic to $\mathbb{R}^2$. Let the small ball $D(0)=B^{2}(p_0)$ expand inside the annular domain $D$ toward the opposite side of the cylinder. At the critical point $p_1$ of the  function $h(p):= (dist_{M^2}(p,p_0))^{2}-\delta^{2}$, the boundary $\partial {D(t)}$ develops a singularity where two boundary points become glued together at the critical point $p_1$ . (Compare this with Step 1 in the proof of Theorem~\ref{TA}, which shall be provided later after Example~\ref{ExD}.)

 The sector domain $\Lambda$ then appears around $p_1$. See Figure~\ref{F5.1}(b). Indeed, $D(t)$ is the sublevel set of $h$. After the critical time $t>h(p_1)$, the domain $D(t)$ continues to expand until it reaches the two boundary circles $\Gamma_{+} \cup \Gamma_{-}=\partial{D}$. Hence, at the end of the deformation, $D(b)=D$.\end{ex}

\begin{figure}[!htb]
\centering
\begin{minipage}[t]{0.55\textwidth}
\centering\includegraphics[scale=0.4]{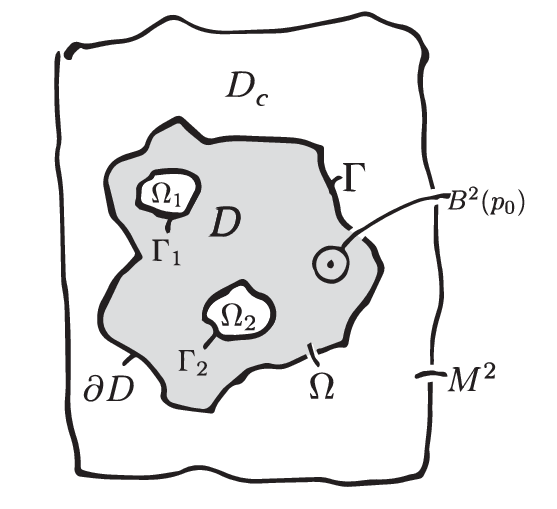}

(a) the given Lipschitz domain $D$ with\\\hspace{-1.9em}two holes $D:=\Omega\setminus(\Omega_1\cup\Omega_2)$
\end{minipage}
\hspace{-1em}
\begin{minipage}[t]{0.45\textwidth}
\centering
\includegraphics[scale=0.4]{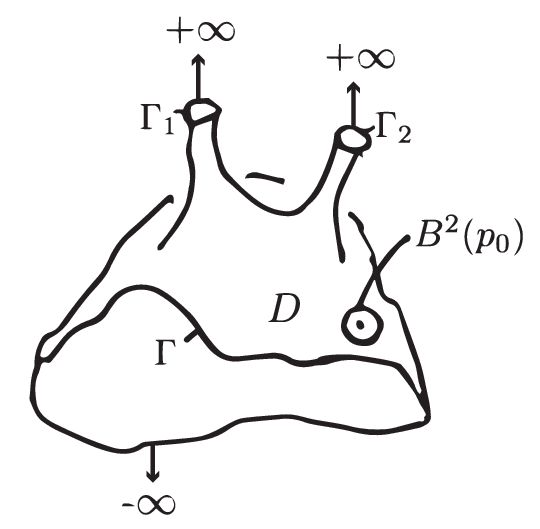}

(b) 1st step to visualize $C^0$-deformation\\\hspace{-3em}on a two-hole domain $D$
\end{minipage}
\caption{Example~\ref{ExD}.}\label{F5.2}
\end{figure}
\begin{ex}\label{ExD}
Consider a domain $D$ in 
 a plane $M^2$ with two holes, namely $D:= \Omega \setminus$ {$\Omega_{1} \cup \Omega_{2}$} where $\Omega$ is a large disk, and
$\Omega_{1}, \Omega_{2}$ are two small disks compactly contained in $\Omega$.
See Figure~\ref{F5.2}(a).
Let $\Gamma:=\partial {\Omega}$, $\Gamma_{1}:=\partial {\Omega_{1}}$, and $\Gamma_{2}:=\partial {\Omega_{2}}$.
As in Example~\ref{ExC}, we shall deform a given small  2-ball $B^{2}(p_0)=D(0)$ inside $D$ into the entire domain $D=D(b)$.
Note that the so-called 2-ball is exactly a disk.
To visualize the deformation, pull the outer boundary circle $\Gamma$ downward along the negative z-axis, while pulling $\Gamma_{1}$ and $\Gamma_{2}$ upward along the positive z-axis.
See Figure~\ref{F5.2}(b). Then the domain $D$ resembles a tree with a vertical trunk and two branches stretching upward into the sky. See Figure~\ref{F5.2cd}(c).
\setcounter{figure}{1}
\begin{figure}[!htb]
\centering
\begin{minipage}[t]{0.55\textwidth}\centering\includegraphics[scale=0.45]{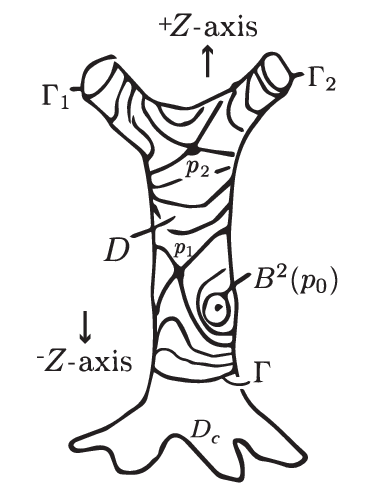}

(c) deforming $B^2(p_0)$ to\\ \qquad a two-hole domain $D$
\end{minipage}
\hspace{-3em}
\begin{minipage}[t]{0.45\textwidth}
\centering\includegraphics[scale=0.4]{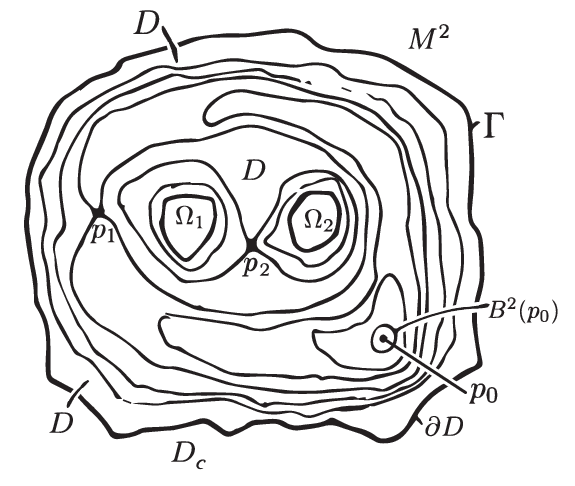}

(d) restore the deformation on the\\~~~~~given planar $D$ with two holes
\end{minipage}
\caption{Example~\ref{ExD}.}\label{F5.2cd}
\end{figure}

Starting from the small disk $B^{2}(p_0)=D(0)$, the sublevel sets $D(t)$ expand along the trunk as in Example~\ref{ExC}. The given small 2-ball $B^{2}(p_0)=D(0)$ on the trunk expands until a first singular transition occurs at the critical point $p_1$ of the distance function $h=h(p):= (dist_{M^2}(p,p_0))^2 - \delta^2$. At the critical level $t=h(p_1)$, two points on the $\partial{D}$ become  glued together at $p_{1}$. Compare Figure~\ref{F5.2cd}(c). A sector domain appears around $p_1$.\\

For $t>h(p_1)$, the domain  $D(t)$ continues expanding with its upper boundary curve approaching toward the two upper branches. A second singular transition occurs at another critical point $p_2$ where $t=h(p_2)$. See Figure~\ref{F5.2}(c).
Again, two boundary points of $D(t)$  become glued together, and 
another sector domain appears around $p_2$.
After passing the second critical level, the expanding domains $D(t)$ eventually reach all three boundary components $\Gamma$, $\Gamma_{1}$ and $\Gamma_{2}$. Hence, $D(b)=D$ at the end of the deformation. It is interesting to see how $D(t)$ deforms in the originally given planar domain. We draw Figure~\ref{F5.2cd}(d) to restore the expanding process of $D(t)$.

\end{ex}

We now present a general process of the $C^0$-deformation from a small $n$-ball to a given Lipschitz domain $D$.

\renewcommand{\proofnamefont}{\bfseries}
\begin{proof}[Proof of Theorem~\textup{\ref{TA}}]
For the first part of Theorem~\ref{TA}, the proof has already been obtained from the continuity theorems, Theorems~\ref{T4.1} and \ref{T4.2}.  It remains to prove the existence of the required $C^0$-deformation of domains. 

\emph{Step~$1$.} Given any $p_{0} \in D$, and a small n-ball $B^{n}(p_0)$ centered at $p_0$ with radius $\delta >0$. Consider $h \colon D \xrightarrow{C^{\infty}} \mathbb{R}^{1}$ defined by $h(p) := (\operatorname{dist}_{M^{n}}(p,p_{0}))^{2}- \delta ^2$ for all $p \in D$. We may assume by a slight modification of $h$ that $h$ is \emph{nondegenerate} in the sense that the Hessian $\operatorname{Hess}(h)$ of $h$ has no zero eigenvalue. This is possible since the set of nondegenerate functions is dense in $C^{\infty}(D,\mathbb{R})$, which is equipped with the compact-open topology. Let $\{ p_{0},p_{1},p_{2},\ldots,p_{\mu} \}$ be the set of critical points of $h$ in $D$. Again, by a slight perturbation, we assume no $p_{j}$ with $h(p_{j}) = h(p_{j+1})$, $j = 0,1,\ldots,\mu-1$. Let
\[
  h(p_{0}) < h(p_{1}) < h(p_{2}) < \cdots < h(p_{\mu}).
\]
Define
\[
  D(t) := \{ p \in D \mid h(p) < t \}
\]
the ``sublevel set" of $h$ at $t$. Then $D(0)$ = $n$-ball $B^{n}(p_{0})$. Write $t_{j} := h(p_{j}),~ j=1,2,\ldots \mu$.  Consider the sequence
\[
  0 <  t_{1} < c_{1} < t_{2} < \cdots < c_{\mu-1} < t_{\mu}<c_{\mu} < b.
\] We will show that
\[
  \mathcal{D} := \{ D(t) \subset M^{n} \mid t \in [0,b] \}
\]
is a monotone $C^{0}$-deformation on $M^{n}$ where $D(0) = B^{n}(p_{0})$, $D(b)$ is the given Lipschitz domain $D$ in $M^{n}$, where $b := \sup \{ h(p) \mid p \in D \}$.\\

\emph{Step~$2$.} To ensure that
all the critical points $p_j$ of $h$ are sitting in an interior domain $D_0$ compactly contained in $D$, we have to modify the metric of $D$ around its boundary by stretching $\partial D$ sufficiently far away, so that $\partial D$ is disjoint from the closure of the sublevel set $D(c_{\mu}) := \{ p \in D ; h(p)<c_{\mu} \}$, where $t_{\mu}<c_{\mu}<b$. In fact, let $N_{\epsilon}$ be a  tubular neighborhood of $\partial D$ in $D$ in the begining. Namely, let $N_{\epsilon}$ := $\{ p \in D ; \zeta(p):=\operatorname{dist}_{M^{n}}(p, \partial D) < \epsilon \}$, $\epsilon > 0$ a small number. We remodel the metric $g_{ij}$ of $M^{n}$ in $N_{\epsilon}$ by
\[
  \widetilde{g}_{ij}(p) := (\Lambda \cdot (\zeta(p)-\epsilon)^{2} + 1) g_{ij}(p)
\]
with $\Lambda$ sufficiently large. Then we can have $h(N_{\epsilon/2}) > c_{\mu}$, and $\{p_{j}; j=1,2,\ldots,\mu\}\subset D_0$ := $D \setminus \overline{N_{\epsilon /2}}$.\\

\emph{Step~$3$.} Classical Morse theory tells us that $D$ is homotopic to the cell-complex $C_{h}$, constituted by the family $\{ \sigma_{j} \mid j = 0,1,2,\ldots,\mu \}$ of cells $\sigma_{j}$ with $\gamma_{j} := \dim \sigma_{j}$ = the index of the Hessian $\operatorname{Hess}(h)$ of $h$ at $p_{j}$, $j = 0,1,2,\ldots,\mu$. Furthermore, it says that any two sublevel sets $D(t)$ with $t_{j-1} < t < t_{j}$ are homotopic to each other along the integral curves of the gradient $\nabla h$ of $h$ on $D$. At $t_{j} = h(p_{j})$, consider $\widehat{D}_{j} := \sigma_{j} \bigcup_{\varphi_{j}} D(c_{j-1})$, where $\partial \sigma_{j}$ is diffeomorphic to the $(\gamma_{j}-1)$-dimensional sphere $S^{\gamma_{j}-1}$, and it is glued to $\partial D(c_{j-1})$ by a ``diffeomorphic into" map $\varphi_{j} \colon \partial \sigma_{j} \to \partial D(c_{j-1})$. (See Figures~\ref{F5.3} and Figures~\ref{F5.4} for a special case.) As $t$ increases from $c_{j-1}$ to $t_{j}$, $\widehat{D}_{j}(t) := \sigma_{j}(t) \bigcup_{\varphi_{j}(t)} D(t)$ varies to become $D(t_{j})$ with the $\gamma_{j}$-cell $\sigma_{j}(t)$ shrinked to the point $p_{j}$. To visualize the gluing, see also the following Step~4. Remark that $D(t_{j})$ is a quasi-Lipschitz domain in $M^{n}$ with glued point $p_{j}$, where the model space $\Omega$ is $D(c_{j-1})$ and the preglued set $\alpha^{k}$ = $\partial \sigma_{j} \subset \partial D(c_{j-1}) = \partial \Omega$. Check Definition~\ref{D2.5}, if needed. The dimension $k = k(j)$ of $\alpha^{k}$ depends on $j$, and equals $\gamma_{j}-1$, i.e., $k = \gamma_{j}-1$.\\

\emph{Step~$4$.} Clearly, the index $\gamma_{0}$ of $\operatorname{Hess}(h)$ is zero at $p_{o}$, since $h$ is nondegenerate and attains minimum at $p_{0}$. Thus the sublevel set $D(0)$ is diffeomorphic to an $n$-ball $B^{n}(p_{0})$. We illustrate the arguments developed in Step~3 by a simple case that $n = 3$, $j=1,~ p_{j} = p_{1}$, and $\gamma_{1} := \dim \sigma_{1} = 1$. See Figure~\ref{F5.4}(a) and Example~\ref{ExF}. Remark that $D(0)$ is a 3-ball, which is illustrated in Figure~\ref{F5.4}(a) by the exterior 3-ball $B^{3}_{c} \cup {p_\infty}$ as before in Example~\ref{ExA}. 

Here $\sigma_{1}$ is an open interval $(q'_{0},q'_{1})$ and $\partial \sigma_{1}$ is the set of two points $\{ q'_{0}, q'_{1} \}$. Let $\widehat{D}_{1} := \sigma_{1} \bigcup_{\varphi_{1}} D(0)$, which looks like a 3-ball $D(0)$ attached with a $1$-hande $\sigma_{1} = (q'_{0},q'_{1})$ with $\partial \sigma_{1} = \{ q'_{0}, q'_{1} \}$ glued on the boundary $\partial D(0)$ of the 3-ball $B^{3}_{c} \cup {p_\infty}$. See Figure~\ref{F5.4}(a). Here the model space $\Omega$ is the 3-ball $D(0)$, and the preglued set is $\alpha^{0} := \partial \sigma_{1} = \{ q'_{0}, q'_{1} \}$. As $t$ increases from $0$ to $t_{1} := h(p_{1})$, the two points $q'_{0}$ and $q'_{1}$ get closer and closer to each other, and finally glued together at the critical point $p_{1}$ of $h$. See Figure~\ref{F5.4}(b), Figure~\ref{F5.4}(c) and Example~\ref{ExF}. By $h$ nondegenerate at $p_{1}$, the shaded area $D(t_{1})$ is locally a ``sector domain" around $p_{1}$, and hence the sublevel set $D(t_{1})$ is a quasi-Lipschitz domain in $M^{3}$. The topological structure of $D(t)$ changes at $t = t_{1}$. When $t$ increases further from $t_{1}$ to $c_{1}$, $D(t)$ expands along $\nabla h$ to become $D(c_{1})$, and the topology of $D(t)$ changes from a 3-ball to a solid torus $D(c_{1})$.\\

%\begin{figure}[!h]
%\centering
%\includegraphics[scale=0.5]{fig/fig501.eps}
%\caption{Example~\ref{ExC} --- construction of $C^0$-deformation} \label{F5.1}
%\end{figure}

\emph{Step~$5$.} In summary, we have constructed a $C^{0}$-deformation $\mathcal{D}$ of enlarging quasi-Lipschitz domains in $M^{n}$. In fact, each domain $D(t)$ has smooth boundary, except at $t = t_{j} := h(p_{j})$ for some $j = 1,2,\ldots,\mu$, where a glued point $p_{j}$ of $D(t_{j})$ appears with preglued set $\alpha^{k} = \partial \sigma_{j}$, and $D(t_{j})$ is a quasi-Lipschitz domain in $M^{n}$, satisfying the conditions of Definition~\ref{D2.5}. We note that $D(t_{j})$ is a sector domain around $p_{j}$, because : (1)~$h$ is nondegenerate at $p_j$; (2)~$D(t_{j})$ tangents to the portion of the tangent space $T_{p_{j}}(M^{n})$ at $p_{j}$ on which $\operatorname{Hess}(h)$ is negative. This is due to the well-known Morse lemma (see John Milnor: Morse theory).\\

\emph{Step~$6$.} As $t > h(p_\mu)$, the sublevel set $D(t)$ of $h$ continuously expands, and it reaches the boundary $\partial D$ at $t=b$ with the modification of metric $g_{i,j}$ in Step 2, since $b = sup\{h(p); p \in D$\}. 
Hence we obtain the required $C^{0}$-deformation $\mathcal{D} := \{ D(t); t\in [0,b]\}$ with $D(0) = B^{n}(p_{0})$, and $D(b)$ is the given quasi-Lipschitz domain $D$ in $M^{n}$. This completes the proof of Theorem~\ref{TA}. 
\end{proof}

In Examples~\ref{ExA} and \ref{ExB} of Section 2, we illustrate  two typical types of quasi-Lipschitz domains: set-glued and point-glued. Now we realize the process of topological changes of the quasi-Lipschitz domains in the $C^0$-deformations around the critical points of the  distance function $h$, based on  these two types of quasi-Lipschitz domains.

\begin{ex} \label{ExE}
Consider $n = 3$, $p_{j} = p_{1}$, at which the index is $2$, i.e., the corresponding cell $\sigma_{1}$ has dimension $\gamma_{1} = 2$, and $\sigma_{1} \approx $ 2-disk. Here the sample space $\Omega$ is the sublevel set $D(0) \approx B^{3}(p_{0})$ with the pre-glued set $\alpha^{k} = \alpha^{1} = \partial \sigma_{1}$ = the circle $S^{1} \subset \partial D(0) \approx \partial B^{3}(p_{0}) \approx$ a $2$-sphere $S^{2}$. As in Example~\ref{ExA}, we exchange the interior and exterior of $B^{3}(p_{0})$, by the standard inversion, letting $p_{0} = p_{\infty}$ at infinity (see Figure~\ref{F5.3}(a)). In Figure~\ref{F5.3}(a), $\widehat{D}_{1}$ = the $2$-disk $\sigma_{1}$ glued to $B^{3}(p_{\infty})$ on its boundary $\partial B^{3}(p_{\infty})$. As $t$ increases from $0$ to $t_{1}$, the 2-disk $ \sigma_{1}(t) $ shrinks to the point $p_1$ (see Figures~\ref{F5.3}(b) and \ref{F5.3}(c)). Finally, we obtain $D(t_{1})$ as the whole $3$-space $\mathbb{R}^{3}$, compactified by adding the point $p_{\infty}$ at infinity, from which two small balls $B_{1}$ and $B_{2}$ are dug out, while $\overline{B}_{1} \cap \overline{B}_{2} = \{p_1\}$.
\begin{figure}[!h]
\centering
\includegraphics[scale=0.5]{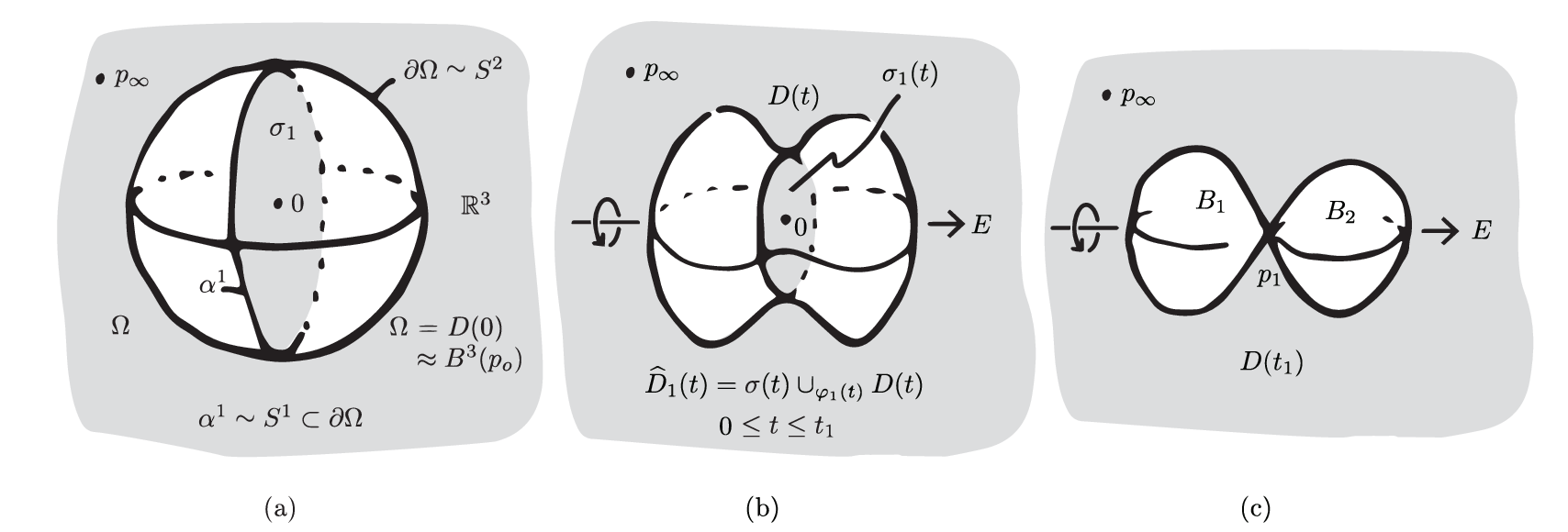}
\caption{Example~\ref{ExE} - index of $h=2$; construction}
of $C^0$-deformation, changing topology from (a) to (c). \label{F5.3}
\end{figure}
\end{ex}

\begin{ex}\label{ExF}
Another example to visualize the construction of the $C^{0}$-deformation in the proof of Theorem~\ref{TA} is relying on Example~\ref{ExB}. Consider $n = 3$, $p_{j} = p_{1}$, at which the index $\gamma_{1}$ of $\operatorname{Hess}(h)$ is $1$, $\sigma_{1} =$ $1$-cell $=$ the interval $(q'_{0},q'_{1})$ with $\partial \sigma_{1}$ = two points $\{ q'_{0}, q'_{1} \}$ glued on $\partial D(0) \approx S^{2}$ (see Figure~\ref{F5.4}(a)). As $t$ increases from $0$ to $t_{1}$, the interval $\sigma_{1}(t) = (q'_{0}(t),q'_{1}(t))$ shrinks into the critical point $p_{1}$ (see Figures~\ref{F5.4}(b) and \ref{F5.4}(c)), which is the glued point of the quasi-Lipschitz domain $D(t_{1})$. Note that $D(t_{1})$ is the whole $3$-space $\mathbb{R}^{3}$, compactified with the point $p_{\infty}$ at infinity, from which a solid torus $V^{2} \times S^{1}$ is dug out, i.e., $D(t_{1}) \approx B^{3} \setminus (V^{2} \times S^{1})$, where $V^2$ is a planar sector. See Figure~\ref{F5.4}(c).
\begin{figure}[!h]
\centering
\includegraphics[scale=0.5]{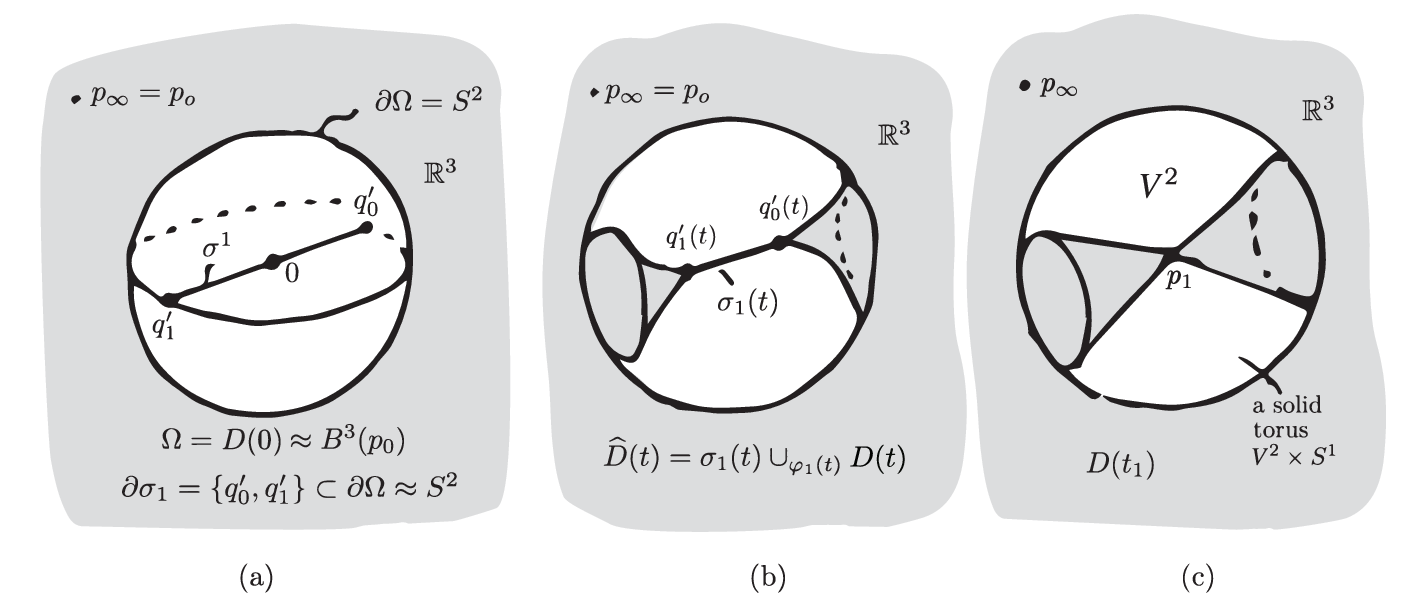}
\caption{Example~\ref{ExF} - index of $h=1$; construction}
of $C^0$-deformation, changing topology from (a) to (c).\label{F5.4}
\end{figure}
\end{ex}

\begin{proof}[Proof of Theorem~\textup{\ref{TB}}]

We prove the global Morse index theorem in the following.

\emph{Step~$1$.} Given $\varphi \in H_{t} = H(D(t))$, we expand
\begin{equation}\label{e5.1} \varphi = a_{1} \: {u}_{1} + a_{2} \: {u}_{2} + \cdots
\end{equation}
in $L^{2}(D(t))$, where the eigenfunctions ${u}_{k} \in
\mathcal {F}_{0}(D(t))$ are the orthonormal basis. Then for any $g \in H_{t}$,
\begin{equation} \label{e5.2}
 B_{L}(\varphi,g)
  = \langle \: a_{1}\: \lambda_{1} {u}_{1}
    + a_{2} \: \lambda_{2} {u}_{2} + \cdots, \; g
    \:\rangle = 0.
\end{equation}
if and only if 
  $a_{i} = 0, \;\forall\, i$ with ${\lambda}_{i}(t) \neq 0$.
The ``if" part is evident, and the ``only if" part  can be seen by letting $g = {u}_{i}$. Thus there exists a (non trivial) Jacobi field on $D(t)$ if and only if $B_L$ on $D(t)$ has a zero eigen value.
Hence, the nulity ${\nu}(t)$ is the multiplicity of the zero eigenvalues of
$B_{L}$.

\bigskip\noindent\emph{Step~$2$.} Let the eigenvalues ${\lambda}_{j}$
on $D(t)$ be written in the form
\begin{equation} \label{e5.3}
  \lambda_{1}
  \leq {\lambda}_{2}
  \leq \cdots
  \leq {\lambda}_{k}
  < 0
  \leq {\lambda}_{k+1}
  \leq {\lambda}_{k+2}
  \leq \cdots
  \to \infty.
\end{equation}
When $D(s)$ enlarges from $0$ to $t$, each of ${\lambda}_{1},
{\lambda}_{2}, \ldots, {\lambda}_{k}$  decreases and continuously passes
through zero   to become negative. We see these by \eqref{e1.13} and Theorem~\ref{T4.2}. It is evident that $\sum_{0 \leq s < t}
{\nu}(t) = k$, since as $s$ varies from $0$ to $t$, there occurs a
Jacobi field on $D(s)$, if and only if some ${\lambda}_{j}(s)$ equals 0,
for $j$ with $1 < j \leq k$, counting the multiplicity.

\bigskip\noindent\emph{Step~$3$.} On the other hand, the Morse index $Ind(D(t))$ is definded by the maiximal dimension of linear subspaces of $H(D(t))$ on which $B_L$ is negatively definite. It means that $Ind(D(t))$ = number of negative eigeinvalues of $B_L$.
By the sequence \eqref{e5.3}, in which the eigenvalues $\lambda_j$ are linearly ordered, it is clear that $Ind(D(t))$ = $k$ = $\sum_{0 \leq s < t}
{\nu}(t)$. This completes the proof.
\end{proof}

In an earlier paper with C.C. Lin \cite{H98}, we introduced the notion of \emph{extremal domains} to initiate the study of Jacobi fields on CMC hypersurfaces $M^n$ immersed in $\mathbb{R}^{n+1}$. In \cite{H24} we obtained  a global Morse index theorem of generalized Lipschitz domains to answer the distribution question of Jacobi fields. The following Theorem~\ref{TC} improves
the results of \cite{H24} by constructing a $C^{0}$-deformation of quasi-Lipschitz domains on the CMC hypersurface $M^n$, for which the global Morse index theorem is also valid.

As in Theorems~\ref{TA} and \ref{TB}, the $C^{0}$-deformation allows topological changes of the  domains while preserving Sobolev continuity and the continuity of eigenvalues along $t$.

\begin{thmm} \label{TC}
Let $M^{n}$ be a hypersurface of constant mean curvature (CMC) immersed in $\mathbb{R}^{n+1}$, $D \subset M^{n}$ be a quasi-Lipschitz domain in $M^{n}$, and $L$ be the stability operator
\begin{equation} \label{e5.4}
  Lf := -\Delta_{M}f - \|B\|^{2} f, \quad \forall\, f \in \mathcal{F}(D)
\end{equation}
extended to the Sobolev sense (see \textup{\cite{H24}} for details), where $\Delta_{M} := (*)_{ii}$ means the Laplacian operator relative to the metric of $M^{n}$, and $\|B\|^{2}$ denotes the length of the second fundamental form of $M^{n}$ in $\mathbb{R}^{n+1}$. Given $p_{0} \in D$. Then there exists a $C^{0}$-deformation $\mathcal{D}$ of enlarging quasi-Lipschitz domains in $D$, where
\[
  \mathcal{D} := \{ D(t) \subset D \mid 0 \leq t \leq b \},
\]
$D(0)=B^{n}(p_{0})$, and $D(b) = D$. The index $\operatorname{Ind}(D)$ of $L$ on $D$ can be counted by \eqref{e1.15}, where $\nu(D(t))$ is the nullity of $L$ on $D(t)$. Indeed, given any $C^{0}$-deformation $\mathcal{D}$ of enlarging quasi-Lipschitz domains in $M^{n}$, the Sobolev continuity in $t$ holds(see \eqref{e1.14}), the eigenvalues $\lambda_{k}(t)$ of $L$ are continuous in $t$, and the Morse index theorem \eqref{e1.15} is valid. Furthermore, instead of the Dirichlet case considered in \eqref{e5.4}, we consider the stability operator
\[
  \widetilde{L}g
  := -\Delta_{M}g - \|B\|^{2}g - \frac{1}{|D|} \int_{D} Lg,
    \quad \forall\, g \in \mathcal{G}(D)
\]
for the case with volume constraint, where
\[
  \mathcal{G}(D)
  := \left\{ g \in \mathcal{F}(D) \Bigm| \int_{D} f \, dM = 0 \right\}.
\]
The previous assertions in the Dirichlet case are also valid for the case with volume constraint.
\end{thmm}

\begin{proof}[Proof of Theorem~\textup{\ref{TC}}]
Since the construction of the $C^{0}$-deformation in the proof of Theorem~\ref{TA} involves only topology and the differentiable structure, it is still applicable to CMC hypersurfaces in $\mathbb{R}^{n+1}$. To show the global Morse index theorem \eqref{e1.15}, we need Sobolev continuity and eigenvalue continuity proved in Theorems~\ref{T4.1} and \ref{T4.2}, as well as the arguments in Sections 4.1 and 4.2 in \cite{H24}. These are now established and available. The proof of Theorem~\ref{TC} is thus completed.
\end{proof}
\renewcommand{\proofnamefont}{\itshape}

%%=================================================================== sec. 6 ==
\section{Concluding remarks} \label{S6}
%%=============================================================================
(i) In 1965, Smale \cite{S65} extended the Morse index theorem on geodesics to smooth domains enlarging in manifolds. The deformation $\mathcal{D}$ was assumed to be $C^{\infty}$, meaning that
all $D(t)\subset M^n$ must be of the same topological type. Hence the initial domain $D(0)$ of Smale's ``$\epsilon$-type" must share the topology same as the final domain $D=D(b)$.
For example, if $M^2$ is the cylinder $S^{1}\times \mathbb{R}^{1}$, and $D$ a ring domain in $M^2$, the $\epsilon$-type initial domain $D(0)$ must be also a ring domain $S^{1}\times (-\varepsilon', \varepsilon')$, which is very thin with $\varepsilon' = O(\varepsilon)$.
Theorem~\ref{TB} in this paper removes the $C^\infty$-deformation restriction  and extends the Morse index theorem to a broader class of $C^{0}$-deformations where the topology of $D(t)$ is allowed to change along $t$. Thus, the initial domain $D(0)$ can be arbitrary, such as a ball $B^n(p_o)$ around a point $p_o$, and the final domain $D(b)$ can be a given quasi-Lipschitz domain of any shape. Furthermore, given two arbitrary quasi-Lipschitz domains $D$ and $D'$ on a manifold $M^n$, we can deform $D$ to a small n-ball $B^n(p_0)$ by shrinking and then deform the n-ball to $D'$ by expanding. Therefore, we obtain the existence of a $C^0$-deformation on which the Morse index theorem holds, with $D(0)=D$, $D(b)=D'$.\\

\noindent (ii) The proof of the Sobolev continuity in the smooth setting is trivial. Following Smale's smooth framework, Frid-Thayer \cite{FT90} proved the eigenvalue continuity based on the Sobolev continuity, and a Morse index theorem in an abstract form. See Remark 3.2 in \cite{H24}.\\

\noindent (iii) The work \cite{H24} of our previous paper was focused on generalized Lipschitz domains deforming in CMC hypersurfaces of $\mathbb{R}^{n+1}$.  The streamlined proof of the Sobolev continuity, i.e. Theorem~\ref{T4.1}, in the present paper improves the complicated arguments illustrated in the proof of Theorem~S of \cite{H24}, and facilitates better understanding. They also help to clarify the tedious reasoning in certain contexts of \cite{H24}.

%%====================================================================== Ref ==

%%============================================ Authors' Affiliation & E-mail ==
\begin{flushleft}
Wu-Hsiung Huang \\ % Author's Name
Department of Mathematics\\ National Taiwan University\\  Taipei 10617, Taiwan \\ % Author's Affiliation
\emph{E-mail address}: 
\texttt{whuang0706@gmail.com}~;\\
\texttt{whuang@math.ntu.edu.tw}\\
ORCID: 0009-0000-5487-3868

\end{flushleft}
\end{document}